\documentclass[final,a4paper,11pt]{article}
\usepackage{graphicx}

\usepackage[T1]{fontenc}
\usepackage[utf8]{inputenc}
\usepackage{amssymb,amsmath,amsfonts,amsthm}

\usepackage[noadjust]{cite}
\usepackage[colorlinks,citecolor=blue]{hyperref}
\usepackage[capitalise,nameinlink]{cleveref}
\usepackage{relsize}
\usepackage[shortlabels]{enumitem}
\usepackage{tikz}
\usepackage{a4wide}
\usepackage{color}

\numberwithin{equation}{section}

\crefname{figure}{Figure}{Figures}

\newcommand{\email}[1]{{\tt #1}}
\newcommand{\R}{\mathbb{R}}
\newcommand{\N}{\mathbb{N}}
\newcommand{\norm}[1]{\|#1\|}

\newcommand{\dist}{\operatorname{dist}}

\newcommand{\mv}{\,\mid\,}
\newcommand{\B}{{\cal B}}

\newcommand{\Sp}{{\cal S}}

\newcommand{\skalp}[1]{\langle #1\rangle}

\newcommand{\xb}{\bar x}
\newcommand{\yb}{\bar y}
\newcommand{\zb}{\bar z}

\newcommand{\inn}{{\rm int\,}}

\newcommand{\gph}{\mathrm{gph}\,}
\newcommand{\dom}{\mathrm{dom}\,}
\newcommand{\tto}{\rightrightarrows}

\newcommand{\conv}{\operatorname{conv}}
\newcommand{\funcsum}{\operatorname{sum}}
\newcommand{\funcdoub}{\operatorname{doub}}

\DeclareMathOperator*{\argmin}{\operatorname{argmin}}

\newcommand{\Gr}{{\rm gph\,}}
\newcommand{\epi}{{\rm epi\,}}
\newtheorem{theorem}{Theorem}[section]
\newtheorem{proposition}[theorem]{Proposition}
\newtheorem{remark}[theorem]{Remark}
\newtheorem{lemma}[theorem]{Lemma}
\newtheorem{corollary}[theorem]{Corollary}
\newtheorem{definition}[theorem]{Definition}

\usetikzlibrary{arrows}

\tikzset{
    punkt/.style={
           rectangle,
           draw=white, very thick,
           text width=9em,
           minimum height=1.5em,
           text centered}
}

\title{Calmness and Calculus: Two Basic Patterns}

\author{%
	Mat\'{u}\v{s} Benko%
	\thanks{%
		University of Vienna,
		Applied Mathematics and Optimization,
		1090 Vienna,
		Austria,
		\email{matus.benko@univie.ac.at},
		\url{https://www.mat.univie.ac.at/\~rabot/group.html}
		}
	\and
	Patrick Mehlitz%
	\thanks{%
		Brandenburgische Technische Universit\"at Cottbus--Senftenberg,
		Institute of Mathematics,
		03046 Cottbus,
		Germany,
		\email{mehlitz@b-tu.de},
		\url{https://www.b-tu.de/fg-optimale-steuerung/team/dr-patrick-mehlitz},
		ORCID: 0000-0002-9355-850X%
		}
	}

\date{}

\begin{document}
\maketitle
 {\bf Abstract.}
 We establish two types of estimates for generalized derivatives
 of set-valued mappings which carry the essence of two basic patterns
 observed throughout the pile of calculus rules.
 These estimates also illustrate the role of the essential
 assumptions that accompany these two patters, namely
 {\em calmness} on the one hand and {\em (fuzzy) inner calmness*} on the other.
 Afterwards, we study the relationship between
 and sufficient conditions for the various notions of (inner) calmness.
 The aforementioned estimates are applied in order to recover 
 several prominent calculus rules for tangents and normals as well as generalized 
 derivatives of marginal functions and compositions as well as Cartesian products 
 of set-valued mappings under mild conditions.
 We believe that our enhanced approach 
 puts the overall generalized calculus into some other light.
 Some applications of our findings are presented which exemplary address necessary optimality
 conditions for minimax optimization problems as well as the calculus related to
 the recently introduced semismoothness* property.
 


\section{Introduction}\label{sec:introduction}

Mathematical analysis as a whole revolves around the invention of differentiation,
which has been gradually extended to plenty of more intricate environments,
such as (possibly infinite-dimensional) spaces of various topological structure
(Banach spaces, locally convex spaces, etc.).
On the other hand, even finite-dimensional optimization,
besides numerous more involved areas of interest,
clearly shows that the objects which have to be analyzed are not restricted
to single-valued smooth mappings.
The need to extend differentiation beyond this setting is,
arguably, most successfully implemented by convex analysis,
particularly, by the convex subdifferential,
see the standard monograph of Rockafellar \cite{Ro70}.
Nevertheless, the convex framework is also too limited and not suitable
at times, as can be seen, for instance, when equilibrium problems are under consideration.
This gives rise to what we call {\em generalized differentiation},
which requires no convexity at all.

In order to be able to apply such generalized differentiation
in a reasonable manner, basic {\em calculus rules} for dealing
with, e.g., compositions and sums of functions should be available.
Perhaps the most natural extension of the convex subdifferential,
the so-called {\em regular (or Fr\'{e}chet) subdifferential}, however,
does not obey suitable rules, see \cite{MordukhovichNamYen2006}, and the same applies
to its primal counterpart, the so-called {\em subderivative}.
Among many attempts to overcome this, one of the most prominent
has been achieved by Mordukhovich \cite[Section~4]{Mordukhovich1976},
who introduced the {\em limiting subdifferential} which comes with
full calculus at hand.

Ever since, the limiting subdifferential as well as the related normal cone
and the associated coderivative have occupied a central place of generalized differentiation
and modern variational analysis.
In particular, the limiting notions play a crucial role
in designing suitable necessary optimality conditions
as well as in the characterization of Lipschitzian behavior of set-valued mappings.
Many prominent researchers have contributed to their development,
and we refer to the two recent monographs of Ioffe \cite{Io17}
and Mordukhovich \cite{Mo18} for detailed information as well as to \cite[Section~1.4]{Mo06}
for a brief sketch visualizing the historical development of generalized differentiation.
Using the recently introduced {\em directional} limiting constructions,
see \cite{GinchevMordukhovich2011,Gfr13a},
the underlying analysis can be refined significantly,
as demonstrated by numerous works authored or co-authored by
Gfrerer \cite{BeGfrOut19,BeGfrOut18,Gfrerer2014,GfrKl16,GfrOut16,GfrOut2019}.

Speaking of calculus rules, we address not only the results
concerning subderivatives and subdifferentials of functions,
but also the parallel results for tangents and normals to sets
on the one hand, and those for graphical derivatives and coderivatives
of set-valued mappings on the other.
Throughout the whole calculus, however, we observe that most of the rules
follow one (or both) of {\em two basic patterns}.
Moreover, each of these two calculus patterns comes with a characteristic
essential assumption given by a certain (Lipschitz) continuity-type requirement for
a suitable surrogate set-valued mapping.

The first pattern, represented e.g.\ by the \emph{intersection rule},
is more prominent and, thus, also much more studied.
In the monographs \cite{RoWe98,Mo06}, the essential assumption for this pattern
is the so-called \emph{Aubin property}, which is perhaps the most famous
extension of Lipschitz continuity to set-valued mappings and
which can be expressed via the so-called \emph{Mordukhovich criterion},
see \cref{Sec:SC} for details.
Later it has been discovered that the Aubin property can be relaxed to
{\em calmness}, see \cite{HenJouOut02,HenOut05,IofOut08},
which is an {\em outer} Lipschitzian concept in its nature,
related to outer semicontinuity.
We would like to mention \cite[Proposition~4.1]{GfrOut16},
which presents a rather general, calculus-like estimate.
It seems that this result contains the essence of this pattern
and that it demonstrates the role of calmness best.
Indeed, unlike the standard calculus rules,
where the focus is on the estimate and calmness is just a tool
that enables it, \cite[Proposition~4.1]{GfrOut16}
answers the opposite question, namely: 
``Having a calm mapping, what can be said about its generalized derivatives?''
Note that \cite[Proposition~4.1]{GfrOut16} in fact deals with
metric subregularity, which is the ``inverse equivalent'' of calmness,
just like metric regularity is the inverse equivalent of the Aubin property,
see \cite[Chapter 3]{DoRo14}.
For more information about calmness and metric subregularity, we further refer the reader to \cite{DoRo04,Iof79a,Io17,Mo18}.

The second pattern, represented e.g. by the \emph{sum rule} (for sets),
is certainly less developed. Here, it turns out that the essential role
is played by conditions of {\em inner-type}.
Indeed, the estimates for the limiting constructions
are known to hold under {\em inner semicontinuity} and {\em inner semicompactness},
see \cite{IofPen96,Mo06}.
Moreover, the estimates for primal and directional limiting objects
were recently shown to hold under {\em inner calmness} and {\em inner calmness*}
in \cite{BeGfrOut19,Be19}.
While inner calmness* was newly defined in \cite{Be19},
inner calmness can be found under other names,
such as, e.g., \emph{Lipschitz lower semicontinuity}, see \cite{KlKum02,KlKu15}, or 
\emph{recession with linear rate}, see \cite{CiFaKru18,Io17}, as well.
We refer to \cite{CiFaKru18} for a comprehensive overview of this and
related notions.
Let us also mention the stronger concept of
Lipschitz lower semicontinuity*,
recently introduced in \cite{CaGiHePa19}
and motivated by a relaxation of the Aubin property
from earlier works of Klatte \cite{Kl87,Kl94}.
We point out that calmness as well as inner calmness*
are automatically satisfied by polyhedral set-valued mappings,
see \cite[Proposition~1]{Rob81} and \cite[Theorem 3.4]{Be19}.

The aim of this paper is two-fold.
First, we want to establish an inner-type analogue of \cite[Proposition~4.1]{GfrOut16}
and then show how these two results, proposed in \cref{The : IC*calc,The : Ccalc},
translate into several standard calculus rules.
Particularly, the chain rule for set-valued mappings demonstrates this nicely since both
patterns are used for its derivation.
We always consider the primal construction (based on tangents) as well as the three dual ones 
(based on regular, limiting, and directional limiting normals), and we even obtain some new estimates
for primal as well as regular dual objects under {\em fuzzy inner calmness*},
a relaxation of inner calmness*.
The second purpose of the paper is to clarify the relations between the various concepts of
(inner) calmness we are going to exploit in our analysis.
Particularly, we focus on verifiable sufficient conditions for these properties.
Alongside the Aubin property, we utilize the {\em isolated calmness} property
and the \emph{first-order sufficient condition for calmness} (FOSCclm for short).
These three conditions can be neatly expressed via generalized derivatives,
namely the limiting coderivative, graphical derivative,
and directional limiting coderivative, respectively.
The inverse equivalents of isolated calmness and FOSCclm
are strong metric subregularity, see \cite{DoRo14}, and Gfrerer's 
first-order sufficient condition for metric subregularity (FOSCMS for short), respectively, see \cite{Gfr13a}, \cite[Corollary 1]{GfrKl16}, and
\cite[Section~2.2]{GfrOut16}.
We note that sufficient conditions for calmness or metric subregularity
have been intensively studied in the literature, see e.g.\
\cite{BenkoCervinkaHoheisel2019,FabHenKruOut10,HenJouOut02,HenOut05,IofOut08,YeZhou2018}. For calmness notions of inner-type, the obtained results are new.
Particularly, in \cref{The:Equiv_GD_IC*fuzzy}, we show that fuzzy inner calmness* is in fact equivalent
to validity of the corresponding calculus estimate for primal objects.

The previous discussion suggests that our findings have the potential to be useful in numerous situations.
Nevertheless, we only detail two interesting applications.
The first one shows how the new rules regarding the regular normal cone
can be utilized to infer optimality conditions for hierarchical optimization
problems of minimax-type, see \cite{Danskin1966,DemyanovMalozemov1974}.
Similarly, these estimates can be used for the derivation of new optimality conditions
for equilibrium-constrained mathematical programs, see e.g.\ \cite{AdamHenrionOutrata2018,GfrererYe2017,GfrererYe2020}.
The second example employs the estimates for directional limiting coderivatives
to study the newly introduced property of \emph{semismoothness*}
which is instrumental in the construction of Newton-type methods for
the numerical solution of generalized equations, see \cite{GfrOut2019}.
 
The remainder of the manuscript is organized as follows: 
In \cref{sec:preliminaries}, we recall the underlying concepts from set-valued and variational 
analysis which are utilized in the paper. \Cref{Sec:Main} is dedicated to the conception of the two aforementioned calculus patterns.
Afterwards, we study the relationship between the various exploited calmness-type conditions and sufficient conditions for their validity in \cref{Sec:SC}.
\Cref{sec:recovering_calculus} provides a collection of standard calculus rules,
where we elaborate on Cartesian products of set-valued mappings in more detail.
\Cref{sec:applications} deals with the aforementioned applications.
The paper closes with the aid of some concluding remarks in \cref{sec:conclusions}.

\section{Preliminaries}\label{sec:preliminaries}

In this section, we provide the essentials of variational
analysis and generalized differentiation which can be found in the monographs \cite{AubinFrankowska2009,Mo06,Mo18,RoWe98}
and the paper \cite{BeGfrOut19}.

\subsection{Basic Concepts from Variational Analysis}\label{sec:basic_notation}

We use $\N$ and $\R$ to denote the sets of natural and real
numbers, respectively.
Throughout the paper, we equip $\R^n$, $n\in\N$, with the 
Euclidean inner product $\skalp{\cdot,\cdot}\colon\R^n\times\R^n\to\R$ and the
Euclidean norm $\norm{\cdot}\colon\R^n\to\R$.
The associated unit ball and unit sphere will be denoted by $\B:=\{z\in\R^n\,|\,\norm{z}\leq 1\}$
and $\Sp:=\{z\in\R^n\,|\,\norm{z}=1\}$, respectively.
For a non-empty set $\Omega\subset\R^n$, we exploit
\[
	\Omega^\circ:=\{z^*\in\R^n\,|\,\skalp{z^*,z}\leq 0\,\forall z\in\Omega\}
\]
in order to represent the polar cone of $\Omega$ which is always a closed, convex
cone. For arbitrary $\bar z\in\R^n$, we set $\bar z+\Omega=\Omega+\bar z:=\{\bar z+z\,|\,z\in\Omega\}$
for brevity of notation. The distance function $\dist(\cdot,\Omega)\colon\R^n\to\R$ 
of the set $\Omega$ is given by
\[
	\dist(z,\Omega):=\inf\limits_{z^{\prime}\in\Omega}\norm{z-z^{\prime}}\quad\forall z\in\R^n.
\]

Let $\Omega\subset\R^n$ be a set which is locally closed around $\zb\in\Omega$.
The tangent (or Bouligand, contingent) cone to $\Omega$ at $\zb$ is given by
\[
	T_{\Omega}(\zb)
	:=
	\{
		w\in \R^n\,|\, \exists (t_k)\downarrow 0,\,\exists (w_k)\to w\colon\,
		\zb+t_kw_k\in\Omega \, \forall k\in\N
	\}.
\]
Furthermore, we use
\begin{equation*}
	\widehat N_{\Omega}(\zb)
	:=
	T_{\Omega}(\zb)^\circ
\end{equation*}
in order to represent the so-called regular (or Fr\'{e}chet) normal cone to $\Omega$ at $\zb$.
We refer to
\[
	N_{\Omega}(\zb)
	:=
	\{
		z^* \in\R^n\,|\, \exists (z_k)\to \zb,\,\exists (z_k^*)\to z^*\colon\,
		 z_k^*\in \widehat N_{\Omega}(z_k)\,\forall k\in\N
	\}
\]
as the limiting (or Mordukhovich) normal cone to $\Omega$ at $\bar z$.
By definition, $\widehat{N}_\Omega(\bar z)\subset N_\Omega(\bar z)$ holds.
In case where $\Omega$ is convex, we have the relations
\[
	\widehat N_\Omega(\bar z)
	=
	N_\Omega(\bar z)
	=
	(\Omega-\bar z)^\circ,
\]
i.e., the regular and limiting normal cone to $\Omega$ at $\bar z$ amount to the
classical normal cone in the sense of convex analysis.

Finally, given a direction $w\in \R^n$, we denote by
\[
	N_\Omega(\zb;w)
	:=
	\left\{
		z^* \in\R^n\,\middle|\,
		\exists (t_k)\downarrow 0, \exists(w_k)\to w, \exists(z_k^*)\to z^*\colon\,
		 z_k^* \in \widehat N_{\Omega}(\zb+t_kw_k)\, \forall k\in\N
	\right\}
\]
the directional limiting normal cone to $\Omega$ at $\zb$ in direction $w$.
This comparatively new variational object has been introduced in
\cite{Gfr13a,GinchevMordukhovich2011}. In \cite{BeGfrOut19}, main calculus
rules for directional limiting normals are discussed, see \cite[Section~4]{Be19}
and \cite[Section~3]{YeZhou2018} as well.
Amongst others, let us mention that $N_\Omega(\bar z;w)=\varnothing$ holds for
all $w\notin T_\Omega(\bar z)$. Moreover, we have
\[
	N_\Omega(\bar z)
	=
	\widehat N_\Omega(\bar z)
	\cup
	\bigcup\limits_{w\in T_\Omega(\bar z)\cap\Sp}
	N_\Omega(\bar z;w),
\]
and $N_\Omega(\bar z;0)=N_\Omega(\bar z)$ is obvious.

For each $\tilde z\notin\Omega$ and arbitrary $w\in\R^n$, we set
$T_\Omega(\tilde z):=\varnothing$ as well as
$\widehat N_\Omega(\tilde z)=N_\Omega(\tilde z)=N_\Omega(\tilde z;w):=\varnothing$
for completeness.

Let us now mention two elementary results for the calculus of 
tangents and normals addressing changes of coordinates and
product structures.

\begin{lemma}\label{lem:change_of_coordinates}
 Let $g\colon\R^n\to\R^m$ be a continuously differentiable function
 and let $C\subset\R^m$ be locally closed around $g(x)\in C$ for some fixed $x\in\R^n$.
 Finally, assume that the Jacobian $\nabla g(x)$ possesses full row rank $m$.
 Then we have
 \begin{align*}
 	T_{g^{-1}(C)}(x)&=\left\{u\,\mid\,\nabla g(x)u\in T_C(g(x))\right\},\\
 	\widehat N_{g^{-1}(C)}(x)&=\bigl\{\nabla g(x)^\top y\,\mid\,y\in\widehat N_C(g(x))\bigr\},\\
 	N_{g^{-1}(C)}(x)&=\bigl\{\nabla g(x)^\top y\,\mid\,y\in N_C(g(x))\bigr\},\\
 	N_{g^{-1}(C)}(x;u)&=
 			\bigl\{\nabla g(x)^\top y\,\mid\,y\in N_C(g(x);\nabla g(x)u)\bigr\}
 			.
 \end{align*}
\end{lemma}
\begin{proof}
	The formulas for tangents as well as regular and limiting normals follow from
	\cite[Exercise~6.7]{RoWe98}. It remains to verify the formula for directional
	limiting normals. 
	
	In case $n=m$, the regularity of $\nabla g(x)$ extends to a neighborhood $U$ of $x$
	where these matrices are continuously invertible. Using now the available formula
	for the regular normal cone to $g^{-1}(C)$ for all points from $g^{-1}(C)\cap U$, 
	we easily infer the desired result from the definition of
	the directional limiting normal cone. In case $n>m$, we exploit the trick from \cite[Exercise~6.7]{RoWe98} in order to deduce the more general statement from the former
	arguments.	
\end{proof}

\begin{lemma}\label{lem:cartesian_products}
 	For natural numbers $n_i\in\N$, $i=1,\ldots,\ell$,
 	sets $C_i\subset\R^{n_i}$, points $x_i\in C_i$
 	where $C_i$ is locally closed, and directions $u_i\in\R^{n_i}$, we have
 	\begin{align*}
 		T_{\prod_{i=1}^\ell C_i}(x_1,\ldots,x_\ell)
 		&\subset
 		\mathsmaller\prod\nolimits_{i=1}^\ell T_{C_i}(x_i),\\
 		\widehat N_{\prod_{i=1}^\ell C_i}(x_1,\ldots,x_\ell)
 		&=
 		\mathsmaller\prod\nolimits_{i=1}^\ell \widehat N_{C_i}(x_i),\\
 		N_{\prod_{i=1}^\ell C_i}(x_1,\ldots,x_\ell)
 		&=
 		\mathsmaller\prod\nolimits_{i=1}^\ell N_{C_i}(x_i),\\
 		N_{\prod_{i=1}^\ell C_i}((x_1,\ldots,x_\ell);(u_1,\ldots,u_\ell))
 		&\subset
 		\mathsmaller\prod\nolimits_{i=1}^\ell N_{C_i}(x_i;u_i).
 	\end{align*}
\end{lemma}
\begin{proof}
	The formulas for regular and limiting normals are stated in
	\cite[Proposition~6.41]{RoWe98}.
	For the remaining inclusions, we refer the reader to \cite[Proposition~3.3]{YeZhou2018}.
\end{proof}

In case where the sets under consideration possess additional regularity
properties, it is even possible to obtain equality in the formulas for
tangents and directional limiting normals
in \cref{lem:cartesian_products}, see \cite[Section~3]{YeZhou2018}.
Particularly, let us mention that equality is obtained in these formulas
if the product of two sets is considered and one of them 
is convex or equals the graph of a continuously differentiable mapping.

Next, we recall the basic notions of generalized differentiation.
For a set-valued mapping $M\colon\R^m\tto\R^n$, 
$\dom M:=\{y\in\R^m\,|\,M(y)\neq \varnothing\}$
and
$\gph M:=\{(y,x)\in\R^m\times\R^n\,|\,x\in M(y)\}$
denote the domain and the graph of $M$, respectively.
The mapping $M^{-1}\colon\R^n\tto\R^m$ given by 
$M^{-1}(x):=\{y\in\R^m\,|\,x\in M(y)\}$ for all $x\in\R^n$ is referred
to as the inverse of $M$.

Let $(\yb,\xb)\in\gph M$ be a point where
$\gph M$ is locally closed.
In this paper, we will consider four types of generalized derivatives
of $M$ corresponding to the four different cones from above applied to $\gph M$ at $(\yb,\xb)$.
The set-valued mapping $D M(\yb,\xb)\colon\R^m\tto\R^n$, defined by
\[
	DM(\yb,\xb)(v)
	:= 
	\left\{
		u\in\R^n\,\middle|\,(v,u)\in T_{\gph M}(\yb,\xb)
	\right\}
	\quad 
	\forall v\in\R^m,
\]
is called the graphical (or Bouligand) derivative of $M$ at $(\yb,\xb)$.
We refer to the mapping $\widehat D^* M(\yb,\xb)\colon\R^n\tto\R^m$, given by
 \[
 	\widehat D^*M(\yb,\xb)(x^*)
 	:=
 	\left\{
 		y^*\in \R^m\,\middle|\, (y^*,- x^*)\in \widehat N_{\gph M}(\yb,\xb )
 	\right\}
 	\quad
 	\forall x^*\in \R^n,
 \]
as the regular (or Fr\'{e}chet) coderivative of $M$ at $(\yb,\xb)$.
The mapping $D^* M(\yb,\xb)\colon\R^n\tto\R^m$, defined via
\[ 
	D^* M(\yb,\xb)(x^*)
	:=
	\left\{
		y^*\in \R^m \,\middle|\,
		(y^*,- x^*)\in N_{\gph M}(\yb,\xb)
	\right\}
	\quad
	\forall x^*\in \R^n,
\]
is referred to as the limiting (or Mordukhovich) coderivative of $M$ at $(\yb,\xb)$.
Finally, given a pair of directions $(v,u)\in\R^m\times\R^n$, the
mapping $D^* M((\yb,\xb);(v,u))\colon\R^n\tto\R^m$, given by
\[
	D^* M((\yb,\xb);(v,u))(x^*)
	:=
	\left\{
		y^* \in \R^m\,\middle|\,(y^*,-x^*)\in N_{\gph M}((\yb,\xb);(v,u)) 
	\right\}
	\quad
	\forall x^*\in \R^n,
\]
is called the directional limiting coderivative of $M$ in direction $(v,u)$ at $(\yb,\xb)$.
For a single-valued mapping $F\colon\R^m\to\R^n$ and some point $\yb\in\R^m$,
we denote the above derivatives without $\xb=F(\yb)$ for brevity.
In case where $F$ is continuously differentiable at $\yb$, we obtain the relations
$DF(\yb)(v)=\nabla F(\yb)v$ as well as $\widehat{D}^*F(\yb)(x^*)=D^*F(\yb)(x^*)=\nabla F(\yb)^\top x^*$
for arbitrary $v\in\R^m$ and $x^*\in\R^n$.
Particularly, we have $D^*F(\yb;(v,u))(x^*)\neq\varnothing$ if and only if $u=\nabla F(\yb)v$ holds.
In this case, $D^*F(\yb;(v,u))(x^*)=\nabla F(\yb)^\top x^*$ follows.

For an extended real-valued function $f\colon\R^n\to\overline\R$, 
$\epi f:=\{(x,\alpha)\in\R^n\times\R\,|\,f(x)\leq\alpha\}$ denotes its
epigraph. It is well known that $f$ is lower semicontinuous at some point $\bar x\in\R^n$
satisfying $|f(\bar x)|<\infty$ if and only if $\epi f$ is locally closed at $(\bar x,f(\bar x))$.
Fix such a point $\bar x\in\R^n$.
For the purposes of this paper, we avoid the standard definitions
of generalized derivatives of $f$ but introduce them rather
via their characterizations in terms of tangents and normals
to $\epi f$.
The extended real-valued function $\mathrm df(\bar x)\colon\R^n\to\overline\R$, characterized by
\[
	\epi \mathrm df(\bar x) =
	T_{\epi f}(\xb,f(\xb)),
\]
is called the subderivative of $f$ at $\bar x$.
Furthermore, we define the regular and limiting subdifferential of $f$ at $\bar x$ by
means of
\begin{align*}
	\widehat\partial f(\bar x)
	&:=
	\left\{
		x^*\in\R^n\,\middle|\,(x^*,-1)\in\widehat N_{\epi f}(\bar x,f(\bar x))
	\right\},\\
	\partial f(\bar x)
	&:=
	\left\{
		x^*\in\R^n\,\middle|\,(x^*,-1)\in N_{\epi f}(\bar x,f(\bar x))
	\right\},
\end{align*}
respectively. Finally, for each pair $(u,\mu)\in\R^n\times\R$, we refer to
\begin{align*}
    \partial f(\bar x;(u,\mu))
	& := 
	\left\{
		x^*\in\R^n\,\middle|\,(x^*,-1)\in N_{\epi f}((\bar x,f(\bar x));(u,\mu))
	\right\},\\
	\partial^{\infty} f(\bar x;(u,\mu))
	& := 
	\left\{
		x^*\in\R^n\,\middle|\,(x^*,0)\in N_{\epi f}((\bar x,f(\bar x));(u,\mu))
	\right\}
\end{align*}
as the directional limiting subdifferential of $f$ at $\bar x$ in direction $(u,\mu)$
and the singular directional limiting subdifferential of $f$ at $\bar x$ 
in direction $(u,\mu)$, respectively.
For a connection to generalized derivatives of set-valued mappings,
we refer to \cite[Theorem~8.2]{RoWe98}.

\subsection{Stability Notions for Set-Valued Mappings}\label{sec:stability_notions}

Below, we present the definitions of the most important stability concepts
for this paper.
We note that several other related notions
are introduced and studied in \cref{Sec:SC}.
Recall that a sequence $(z_k)$
is said to converge to $z$ from a direction $w$ whenever there are
sequences $(t_k)\downarrow 0$ and $(w_k)\to w$ such that $z_k=z+t_kw_k$ holds for
all $k\in\N$. 

We start by recalling the definitions of inner semicompactness and
inner calmness*.
\begin{definition}\label{def:ICandC}
 	Let $M\colon\R^m\tto\R^n$ be a set-valued mapping and fix $\yb\in\dom M$.
 	Furthermore, let $\Omega\subset\R^m$ and as well as a
 	direction $v\in\R^m$ be arbitrarily chosen.
 	\begin{itemize}
  	\item[(i)] We say that $M$ is \emph{inner semicompact} at $\yb$ w.r.t.\ $\Omega$
  		if for each sequence $(y_k)\subset\Omega$ converging to $\yb$, we find
  		$\bar x\in\R^n$ and a sequence $(x_k)$ such that $(x_k)\to\bar x$ and
  		$x_k\in M(y_k)$ hold along a subsequence.
  		In case where this property only holds for all sequences $(y_k)$ converging
  		to $\yb$ from $v$, we call $M$ inner semicompact at $\yb$ w.r.t.\ $\Omega$
  		in direction $v$.
  	\item[(ii)] We say that $M$ is \emph{inner calm*} at $\yb$ w.r.t.\ $\Omega$ if 
  		there is some $\kappa>0$ such that for each
  		sequence $(y_k)\subset\Omega$ converging to $\yb$, we find $\bar x\in\R^n$
  		and a sequence $(x_k)$ such that $\norm{x_k-\bar x}\leq\kappa\norm{y_k-\yb}$
  		and $x_k\in M(y_k)$ hold along a subsequence.
  		The infimum $\bar\kappa$ over all constants with this property is called
  		modulus of inner calmness* at $\yb$ w.r.t.\ $\Omega$.
  		In case where the above property only holds for all sequences $(y_k)$ converging to
  		$\yb$ from $v$, we call $M$ inner calm* at $\yb$ w.r.t.\ $\Omega$ 
  		in direction $v$.
  		Similarly as above, we introduce the modulus $\bar\kappa_v$ of inner calmness*
  		at $\yb$ w.r.t.\ $\Omega$ in direction $v$.
  	\end{itemize}
  	In case where $\Omega:=\R^m$ can be chosen, we simply omit the expression ``w.r.t.\ $\R^m$''
  	for brevity.
\end{definition}

Clearly, inner calmness* of a set-valued mapping at one point of its domain is stronger
than inner semicompactness since the former provides an additional
information about the rate of convergence. Note that whenever $M\colon\R^m\tto\R^n$
is a set-valued mapping with a closed graph, then $\dom M$ does not need to be closed.
However, fixing $\yb\in\dom M$ where $M$ is inner semicompact w.r.t.\ $\dom M$ yields
closedness of $\dom M$ locally around $\yb$,
see \cite[Lemma 2.1]{Be19}.

Next, we present the definition of a relaxed version inner calmness*.
\begin{definition}\label{def:inner_calmness*_in_fuzzy_sense}
	Let $M\colon\R^m\tto\R^n$ be a set-valued mapping and fix $\yb\in\dom M$,
	a set $\Omega\subset\R^m$ which is locally closed at $\yb$, and a direction $v\in\R^m$.
	We say that $M$ is \emph{inner calm* in the fuzzy sense at $\yb$ w.r.t.\ $\Omega$ 
	in direction $v$} if either $v\notin T_\Omega(\yb)$ holds or if there is some
	$\kappa_v>0$ such that we find a sequence $(y_k)\subset\Omega$ converging to $\yb$
	from $v$, some $\bar x\in\R^n$, and a sequence $(x_k)$ such that $x_k\in M(y_k)$ and
	$\norm{x_k-\bar x}\leq\kappa_v\norm{y_k-\bar y}$ hold for all $k\in\N$.
	The infimum $\bar\kappa_v$ over all constants with this property is called modulus
	of inner calmness* in the fuzzy sense at $\yb$ w.r.t.\ $\Omega$ in direction $v$.
	For $v\notin T_\Omega(\yb)$, we set $\bar\kappa_v:=0$ for the purpose of completeness.
	We say that $M$ is inner calm* in the fuzzy sense at $\yb$ w.r.t.\ $\Omega$ provided it
	is inner calm* in the fuzzy sense at $\yb$ w.r.t.\ $\Omega$ in each direction
	$v\in\Sp$.
	In case where $\Omega:=\R^m$ can be chosen, we simply omit the expression ``w.r.t.\ $\R^m$''
  	for brevity.
\end{definition}

Inner calmness* and inner calmness* in the fuzzy sense have been first defined in \cite{Be19}.
The motivation for introducing the fuzzy inner calmness*
comes from the very natural setting of a certain multiplier mapping
which possesses this property, see \cite[Theorem 3.9]{Be19}.
It is worth noting that fuzzy inner calmness* does not
imply inner semicompactness in general.
In this paper, we extend the calculus rules based on (fuzzy) inner calmness*
from \cite[Section 4]{Be19}, particularly to estimates for regular normals.

Note that each of the ``inner'' conditions from
\cref{def:ICandC,def:inner_calmness*_in_fuzzy_sense} 
implies that the images of the underlying set-valued mapping are non-empty
for all points of the set $\Omega$ that lie near the point of interest $\yb$.
While this can be desirable in some situations, it turns out to be quite restrictive in other settings.
For our purposes, however, we will often consider these properties w.r.t.\
the domain of the mapping, and this does not add any restriction.

Finally, we would like to recall the definition of calmness. 
\begin{definition}\label{def:calmness}
	Let $M\colon\R^m\tto\R^n$ be a set-valued mapping and fix $(\yb,\xb)\in\gph M$ as well as
	a direction $u\in\R^n$. We say that $M$ is \emph{calm} at $(\yb,\xb)$ if there is some
	$\kappa>0$ such that for each sequence $(x_k)\to\xb$ satisfying $x_k\in M(y_k)$ for
	some $y_k$ and all $k\in\N$, we find a sequence $(\tilde x_k)\subset M(\yb)$ such that
	$\norm{x_k-\tilde x_k}\leq\kappa\norm{y_k-\yb}$ holds for sufficiently large $k\in\N$.
	The infimum $\bar\kappa$ over all constants with this property is called the modulus of
	calmness at $(\yb,\xb)$.
	In case where the above property only holds for all sequences $(x_k)$ converging to
	$\xb$ from $u$, we call $M$ calm at $(\yb,\xb)$ in direction $u$.
	Similarly as above, we introduce the modulus $\bar\kappa_u$ of calmness at $(\yb,\xb)$
	in direction $u$.
\end{definition}
 
One can easily check that $M$ is calm at $(\yb,\xb)\in\gph M$ if and only if there
are neighborhoods $U$ of $\xb$ and $V$ of $\yb$ as well as some constant $\kappa>0$ such that
\begin{equation}\label{eq:standard_definition_calmness}
	M(y)\cap U\subset M(\yb)+\kappa\norm{y-\bar y}\B\quad\forall y\in V
\end{equation}
holds, and this corresponds to the classical definition of calmness. 
In order to be consistent with the definition of inner calmness*,
however, we choose to work with the sequential
counterpart of the definition in this manuscript.
It is well known that the neighborhood $U$ can be reduced in such a way that $V$ can be replaced by the whole space $\R^m$, see \cite[Exercise 3H.4]{DoRo14}.
On the one hand, this leads to the equivalent formulation of calmness via
\[
	\dist(x,M(\yb)) \leq \kappa \dist(\yb,M^{-1}(x)) \quad\forall x\in U,
\]
which is precisely the definition of metric subregularity of $M^{-1}$ at $(\xb,\yb)$.
On the other hand, it motivates us to define directional calmness
via directions in the range space $\R^n$ and, in turn,
yields that calmness in a predefined direction $u$
is the inverse counterpart of metric subregularity in direction $u$,
see \cite{Gfr13a}.

In \cref{Sec:SC}, we relate the properties from
above to other prominent Lipschitzian notions.

We conclude this preliminary section with the following
simple result regarding the stability of calmness and inner calmness*
under single-valued calm perturbations.
A single-valued function $\varphi \colon \R^n \to \R^m$
is called calm at $\xb \in \R^n$ if there is a constant $\kappa>0$
such that 
\begin{equation*}
    \norm{\varphi(x_k) - \varphi(\xb)} \leq
	\kappa\norm{x_k-\xb}
\end{equation*}
holds for every sequence $(x_k) \to \xb$ and sufficiently large $k \in \N$.
\begin{proposition}\label{pro : SV_calm_perturb}
    Let $M\colon\R^m\tto\R^n$ be a set-valued mapping,
    let $\varphi_1\colon \R^\ell \to \R^r$, $\varphi_2\colon \R^\ell \to \R^m$
    and $\varphi_3\colon \R^\ell \to \R^n$ be single-valued mappings, and 
    define another set-valued mapping
    $\widetilde M\colon \R^\ell \tto \R^r \times \R^n$ as follows:
    \[
    \widetilde M(q) := \{\varphi_1(q)\} \times \big( M(\varphi_2(q)) + \varphi_3(q) \big)
    \quad\forall q\in\R^\ell.
    \]
    \begin{itemize}
        \item[(i)] If the functions $\varphi_i$, $i = 1,2,3$,
        are continuous (calm) at $\bar q$ and if $M$ is inner semicompact (inner calm*) at
        $\varphi_2(\bar q)$ w.r.t.\ $\dom M$, then $\widetilde M$
        is inner semicompact (inner calm*) at $\bar q$
        w.r.t.\ $\dom \widetilde M$.
        \item[(ii)] If the functions $\varphi_i$, $i = 1,2,3$,
        are calm at $\bar q$ and if $M$ is calm at
        $(\varphi_2(\bar q),\xb) \in \gph M$ for some $\xb$,
        then $\widetilde M$ is calm at
        $\big( \bar q,(\varphi_1(\bar q),\xb + \varphi_3(\bar q))\big) \in \gph \widetilde M$.
    \end{itemize}
\end{proposition}
\begin{proof}
    The proof of the inner semicompactness claim from (i) is simpler, so we only prove
    the claim regarding inner calmness*. 
    For each $i=1,2,3$, let $\kappa_i$ be a calmness constant of $\varphi_i$, and
    let $\kappa>0$ be an inner calmness* constant of $M$ at $\varphi_2(\bar q)$
    w.r.t.\ $\dom M$.
    Furthermore, let us equip product spaces w.l.o.g. with the natural sum norms.
    In order to prove (i), consider a sequence $(q_k) \to \bar q$
    with $(q_k) \subset \dom \widetilde M$.
    This automatically yields $(\varphi_2(q_k))\subset\dom M$.
    The calmness of $\varphi_2$ at $\bar q$ guarantees $\varphi(q_k)\to\varphi(\bar q)$.
   	Thus, by inner calmness* of $M$, we find a sequence $(x_k)$ and a point $\xb$ such that
   	$x_k\in M(\varphi_2(q_k))$ and $\norm{x_k-\bar x}\leq\kappa\norm{\varphi_2(q_k)-\varphi(\bar q)}$
   	hold along a subsequence. Consequently, we find that
    \begin{align*}
        &\norm{(\varphi_1(q_k),x_k + \varphi_3(q_k)) - (\varphi_1(\bar q),\xb + \varphi_3(\bar q))}\\
        &\qquad \leq  \norm{x_k - \xb} + \norm{\varphi_1(q_k) - \varphi_1(\bar q)} 
        	+	\norm{\varphi_3(q_k) - \varphi_3(\bar q)}\\
        &\qquad \leq  \kappa\norm{\varphi_2(q_k) - \varphi_2(\bar q)}
        	+	(\kappa_1 + \kappa_3)\norm{q_k - \bar q}  \\
        &\qquad \leq  (\kappa \kappa_2 + \kappa_1 + \kappa_3)\norm{q_k - \bar q}
    \end{align*}
    holds along a subsequence.
    Since $(\varphi_1(q_k),x_k + \varphi_3(q_k)) \in \widetilde M(q_k)$ holds along a subsequence,
    the claim follows.
   
    Let us now prove (ii). 
    Therefore, let $\kappa>0$ be a calmness constant of $M$ at $(\varphi_2(\bar q),\bar x)$.
    Consider a sequence $((a_k,b_k)) \to (\varphi_1(\bar q),\xb + \varphi_3(\bar q))$ 
    with $(a_k,b_k) \in \widetilde M(q_k)$
    for some $q_k$ and all $k\in\N$, i.e., $(a_k,b_k)=(\varphi_1(q_k),x_k + \varphi_3(q_k))$
    for some $x_k \in M(\varphi_2(q_k))$ and all $k\in\N$ such that $(x_k)\to\xb$.
    The calmness of $M$ now yields the existence of a sequence 
    $(\tilde x_k) \subset M(\varphi_2(\bar q))$
    such that the above estimates with $\xb$ replaced by $\tilde x_k$ hold
    for sufficiently large $k \in \N$ .
    This completes the proof since
    $(\varphi_1(\bar q),\tilde x_k + \varphi_3(\bar q)) \in \widetilde M(\bar q)$ holds.
\end{proof}

\section{On Two Patterns in Calculus} \label{Sec:Main}

From \cref{def:ICandC,def:inner_calmness*_in_fuzzy_sense,def:calmness}, one could perhaps guess
that inner calmness* {\em w.r.t.\ the domain} in fact provides a connection
between the domain of a set-valued mapping $M\colon\R^m\tto\R^n$
and its graph, while calmness connects
the graph of $M$ with its image sets.
Here, we will formalize this conjecture by comparing tangents and normals
to the graph with tangents and normals to the domain
under the inner calmness* on one hand,
and with tangents and normals to the images under the calmness
on the other hand.
Let us start with the former.

\begin{theorem} \label{The : IC*calc}
	Let $M\colon\R^m \tto \R^n$ be a set-valued mapping and let
	$\yb\in\dom M$ be chosen such that $\gph M$ is locally closed around $\{\yb\}\times\R^n$ and $\dom M$ is locally closed around $\yb$.
	Then the following assertions hold.
	\begin{itemize}
		\item[$\mathbf{T}$] \textup{Tangents:} We always have
			\begin{equation*}
				    T_{\dom M}(\bar y)
				\ \supset \
				\bigcup\limits_{\xb\in M(\yb)}\dom DM(\yb,\xb),
			\end{equation*}
			and the opposite inclusion holds true if
			$M$ is inner calm* at $\yb$ w.r.t.\ $\dom M$ in the fuzzy sense.
			Moreover, if $\bar\kappa_v$ denotes the modulus of inner calmness* of $M$ at $\yb$
			w.r.t.\ $\dom M$ in direction $v$ in the fuzzy sense, we in fact have
			\begin{equation*}
			    T_{\dom M}(\bar y)
				\ = \
				\left\{
					v\,\middle|\,
					\inf_{\xb\in M(\yb)}
					\inf_{u\in DM(\bar y,\bar x)(v)}\,
					\norm{u} \leq \bar\kappa_v\norm{v}
				\right\},
			\end{equation*}
			and this holds with $\bar\kappa_v$ replaced by $\bar\kappa$ provided $M$ is even
			inner calm* at $\yb$ w.r.t.\ $\dom M$ with modulus $\bar\kappa$.
		\item[$\mathbf{\widehat{N}}$] \textup{Regular normals:} We always have
			\begin{equation*}
				\widehat{N}_{\dom M}(\yb) \ \subset \
				\bigcap_{\xb \in M(\yb)} \widehat{D}^*M(\yb,\xb)(0),
			\end{equation*}
			and the opposite inclusion holds true if
			$M$ is inner calm* at $\yb$ w.r.t.\ $\dom M$ in the fuzzy sense.
		\item[$\mathbf{N}$] \textup{Limiting normals:} 
			If $M$ is inner semicompact at $\yb$ w.r.t.\ $\dom M$, then we have
			\begin{equation*}
				N_{\dom M}(\yb) \ \subset \
				\bigcup_{\xb \in M(\yb)} D^*M(\yb,\xb)(0).
			\end{equation*}
		\item[$\mathbf{dN}$] \textup{Directional limiting normals:}
			Let $v \in \R^m$ be a fixed direction.
			If $M$ is inner semicompact at $\yb$ in direction $v$ w.r.t.\ $\dom M$, then we have
			\begin{align*}
				N_{\dom M}(\yb;v) 
				\ \subset \				
				\bigcup_{\xb \in M(\yb)}
				&\left[
				    \bigcup_{u \in DM(\yb,\xb)(v)} D^*M((\yb,\xb);(v,u))(0) \right.\\
			    	&\qquad \left.\cup
				    \bigcup_{u \in DM(\yb,\xb)(0) \cap \Sp} D^*M((\yb,\xb);(0,u))(0)
				\right].
			\end{align*}
			Moreover, the union over $DM(\yb,\xb)(0) \cap \Sp$ is superfluous
			if $M$ is inner calm* at $\yb$ in direction $v$ w.r.t.\ $\dom M$
			with modulus $\bar\kappa_v$, and $u \in DM(\yb,\xb)(v)$
			can be chosen with $\norm{u}\leq \kappa \norm{v}$
			for any $\kappa > \bar\kappa_v$ in this case.
	\end{itemize}
\end{theorem}
\begin{proof}
 In order to prove $\mathbf T$, suppose first that there exist $\xb \in M(\yb)$ 
 and $u \in DM(\yb,\xb)(v)$ implying
 the existence of $(t_k) \downarrow 0$ and $(v_k,u_k) \to (v,u)$ with
 $(\yb + t_k v_k,\xb + t_k u_k) \in \gph M$ for all $k\in\N$.
 Particularly, this means $\yb + t_k v_k \in \dom M$ for all $k\in\N$ and, thus, 
 $v \in T_{\dom M}(\yb)$.
 
 On the other hand, consider $v \in T_{\dom M}(\yb)$.
 If $M$ is inner calm* at $\bar y$ w.r.t.\ $\dom M$ in direction $v$ in the fuzzy sense
 with modulus $\bar\kappa_v$, then we find $(t_k) \downarrow 0$
 and $(v_k) \to v$ with $\yb + t_k v_k \in \dom M$ 
 together with $x_k \in M(\yb + t_k v_k)$ for all $k\in\N$ and $\xb\in\R^n$ such that
 \begin{equation} \label{eq : IC*estim}
  \norm{x_k - \xb} \leq \kappa t_k \norm{v_k}
 \end{equation}
 holds for each $\kappa > \bar\kappa_v$ and all $k\in\N$.
 Introducing $u_k:=(x_k-\xb)/t_k$, this means that
 $(\yb + t_k v_k, \xb + t_k u_k) \in \gph M$
 is valid for all $k\in\N$. Additionally, we have
 $\norm{u_k} \leq \kappa \norm{v_k}$
 for all $k\in\N$, and the boundedness of
 $(v_k)$ yields the boundedness of $(u_k)$.
 Thus, we find $u\in\R^n$ such that
 $u \in DM(\yb,\xb)(v)$
 and $\norm{u} \leq \kappa \norm{v}$ hold,
 taking also into account that $\gph M$ is locally closed around each $(\yb,\xb)$
 yielding $(\yb,\xb)\in\gph M$ and, thus, $\xb\in M(\yb)$.
 Hence,
 $\inf_{\xb\in M(\yb)}\inf_{u\in DM(\bar y,\bar x)(v)}\, \norm{u} \leq \kappa\norm{v}$
 follows, and since $\kappa > \bar\kappa_v$ was arbitrary,
 the infimum is also bounded from above by $\bar\kappa_v \norm{v}$.
 The statement about inner calmness* follows immediately from
 $\bar\kappa_v \leq \bar\kappa$ for every direction $v$.

 Next, let us prove the estimate $\mathbf{\widehat{N}}$.
 Keeping the statement $\mathbf{T}$ for tangents in mind,
 it is sufficient to prove that
 \[
 	\left(\bigcup\nolimits_{\xb \in M(\yb)} \dom DM(\yb,\xb) \right)^{\circ}
 	\ = \ 
 	\bigcap\nolimits_{\xb \in M(\yb)} \widehat{D}^*M(\yb,\xb)(0)
 \]
 holds in order to show the estimates for regular normals.
 Exploiting the calculus rules of polarization, this holds whenever
  $(\dom DM(\bar y,\bar x))^\circ=\widehat{D}^*M(\bar y,\bar x)(0)$
 is valid for all $\bar x\in M(\bar y)$. Thus, let us fix $\bar x\in M(\bar y)$.
 
 Pick $y^* \in (\dom DM(\yb,\xb))^{\circ}$
 and consider an arbitrary  pair $(v,u) \in T_{\gph M}(\yb,\xb)$.
 Then we particularly have $v \in \dom DM(\yb,\xb)$,
 which yields $ \skalp{(y^*,0),(v,u)}=\skalp{y^*,v} \leq 0$,
 showing $y^*\in\widehat{D}^*M(\yb,\xb)(0)$.
 Conversely, fix $\tilde y^* \in \widehat{D}^*M(\yb,\xb)(0)$
 and consider $v \in \dom DM(\yb,\xb)$.
 Then we find $u \in DM(\yb,\xb)(v)$, and 
 $\skalp{\tilde y^*,v} = \skalp{(\tilde y^*,0),(v,u)} \leq 0$
 follows which leads to $\tilde y^* \in (\dom DM(\yb,\xb))^{\circ}$.
 
 In order to prove $\mathbf{N}$, consider $y^* \in N_{\dom M}(\yb)$.
 By definition, we find sequences
 $(y_k) \to \yb$ and $(y_k^*) \to y^*$ with
 $y_k^* \in \widehat{N}_{\dom M}(y_k)$ for all $k\in\N$.
 By inner semicompactness of $M$ w.r.t.\ $\dom M$, there exist $\xb\in\R^n$
 and a sequence $(x_k)\to\xb$ with $x_k\in M(y_k)$ for all $k\in\N$
 (at least along a subsequence without relabelling).
 The closedness properties of $M$ guarantee $\xb\in M(\yb)$.
 The estimate $\mathbf{\widehat{N}}$ for regular normals yields, in particular, that
 $ y_k^* \in \widehat{D}^*M(y_k,x_k)(0)$ holds for all $k\in\N$.
 Thus, taking the limit yields $y^*\in D^*M(\yb,\xb)(0)$.
  
 Finally, let us prove $\mathbf{dN}$.
 Therefore, we fix $y^*\in N_{\dom M}(\yb;v)$.
 By definition, we find sequences $(t_k) \downarrow 0$, $(v_k) \to v$, and
 $(y_k^*)\to y^*$ such that $y_k^*\in\widehat{N}_{\dom M}(\bar y+t_kv_k)$ holds
 for all $k\in\N$. Since $M$ is inner semicompact at $\bar y$ w.r.t.\ $\dom M$ in direction $v$,
 we find a point $\bar x\in\R^n$ and a sequence $(x_k)\to\bar x$ satisfying
 $x_k\in M(\bar y+t_kv_k)$ for all $k\in\N$ (at least along a subsequence without relabelling).
 Statement $\mathbf{\widehat{N}}$ yields
 $y_k^*\in \widehat{D}^*M(\bar y+t_kv_k,x_k)(0)$ for all $k\in\N$.
 Now, we distinguish two cases. 
 
 First, assume that $((x_k-\bar x)/t_k)$ is bounded.  
 In this case, there is some $u\in\R^n$ such that $((x_k-\bar x)/t_k)\to u$
 holds along a subsequence (without relabelling). Due to the relation
 $(\bar y+t_kv_k,\bar x+t_k(x_k-\bar x)/t_k)\in\gph M$ for all $k\in\N$, we 
 find $u\in DM(\bar y,\bar x)(v)$ by taking the limit $k\to\infty$.
 Furthermore, due to $y_k^*\in \widehat D^*M(\bar y+t_kv_k,\bar x+t_k(x_k-\bar x)/t_k)(0)$
 for all $k\in\N$, the definition of the limiting directional coderivative yields
 $y^*\in\widehat D^*M((\bar y,\bar x);(v,u))(0)$.
 Note that this argumentation is always possible if $M$ is inner calm* at $\bar y$ in direction
 $v$ w.r.t.\ $\dom M$ with modulus $\bar\kappa_v$ 
 since we have $\norm{(x_k-\bar x)/t_k}\leq\kappa\norm{v_k}$ for all
 $k\in\N$ and each $\kappa>\bar\kappa_v$ in this case for the sequence $(x_k)$ and the point
 $\bar x$ from \cref{def:ICandC}. Taking the limit $k\to\infty$ yields $\norm{u}\leq\kappa\norm{v}$.
 
 Next, we assume that $((x_k-\bar x)/t_k)$ is not bounded. In this case, $t_k/\norm{x_k-\bar x}\to 0$
 needs to hold along a subsequence (without relabelling). Setting 
 $\tilde u_k:=(x_k-\bar x)/\norm{x_k-\bar x}$ for all $k\in\N$, we find some $\tilde u\in\Sp$
 such that $\tilde u_k\to\tilde u$ holds along a subsequence (without relabelling again).
 Observing that we have
 \[
 	\left(\bar y+\norm{x_k - \xb} \frac{t_kv_k}{\norm{x_k - \xb}},
    \bar x+\norm{x_k - \xb} \tilde u_k\right)\in\gph M
 \]
 for all $k\in\N$, $\tilde u\in DM(\bar y,\bar x)(0)$ follows by taking the limit $k\to\infty$.
 Due to
 \[
    y^*_k\in\widehat{D}^*M\left(\bar y+\norm{x_k - \xb} \frac{t_kv_k}{\norm{x_k - \xb}},
    \bar x+\norm{x_k - \xb} \tilde u_k\right)(0)
 \]
 for all $k\in\N$, $y^* \in D^*M((\bar y,\bar x);(0,\tilde u))(0)$ is obtained by taking the limit
 $k\to\infty$.
\end{proof}

Let us recall again that the local closedness of $\dom M$
is automatically satisfied if $M$ is inner semicompact at $\yb$ w.r.t. $\dom M$.

Note that the above theorem extends the results from
\cite[Theorem~4.1]{Be19} and \cite[Theorem~3.2]{BeGfrOut19}.

The formulas in \cref{The : IC*calc} reveal an interesting feature:
Tangents to the domain can be expressed as a domain
of the graphical derivative.
On the other hand, normals to the domain
are estimated via the image (of $0$) of the coderivative.
In other words, the primal constructions preserve the domain
while the dual ones flip it to the image.

Now, let us deal with the other pattern connecting tangents and normals
to images of $M$ and the graph of $M$, respectively, by exploiting calmness.

\begin{theorem}\label{The : Ccalc}
Assume that $M\colon\R^m\tto\R^n$ possesses a locally closed graph
around $(\bar y,\bar x)\in\gph M$.
Then the following assertions hold.
\begin{itemize}
	\item[$\mathbf T$] \textup{Tangents:} We always have
		\[
			T_{M(\bar y)}(\bar x)\ \subset \ DM(\bar y,\bar x)(0),
		\]
		and the opposite inclusion holds true whenever $M$ is calm at
		$(\bar y,\bar x)$.
	\item[$\mathbf{\widehat{N}}$] \textup{Regular normals:} We always have
		\[
			\widehat{N}_{M(\bar y)}(\bar x)\ \supset \ -\dom \widehat{D}^*M(\bar y,\bar x).
		\]
	\item[$\mathbf N$] \textup{Limiting normals:} If $M$ is calm at $(\bar y,\bar x)$
		with modulus $\bar\kappa > 0$, then we have
		\[
			\begin{aligned}
				N_{M(\bar y)}(\bar x)
				\ &\subset \
				\left\{-x^*\,\middle|\, \inf_{y^*\in D^*M(\bar y,\bar x)(x^*)}\,
				            \norm{y^*} \leq \bar\kappa\norm{x^*}\right\}\\
				\ &\subset \
				-\dom D^*M(\bar y,\bar x).
			\end{aligned}			
		\]
	\item[$\mathbf{dN}$] \textup{Directional limiting normals:}
		Let $u\in\R^n$ be a fixed direction.
		If $M$ is calm at $(\bar y,\bar x)$ in direction $u$ with modulus $\bar\kappa_u > 0$,
		then we have
		\[
			\begin{aligned}
				N_{M(\bar y)}(\bar x;u)
				\ &\subset \
				\left\{-x^*\,\middle|\, \inf_{y^*\in D^*M((\bar y,\bar x);(0,u))(x^*)}\,
				            \norm{y^*} \leq \bar\kappa_u\norm{x^*}\right\}\\
				\ &\subset \
				-\dom D^*M((\bar y,\bar x);(0,u)).
			\end{aligned}
		\]
\end{itemize}
\end{theorem}
\begin{proof}
Converting \cite[Proposition~4.1]{GfrOut16} from the setting of metric subregularity
into the calmness setting readily yields equality in the statement $\mathbf T$ 
as well as the estimate
\[
N_{M(\bar y)}(\bar x;u) \subset \{-x^* \mv \exists\, y^* \in D^*M((\bar y,\bar x);(0,u))(x^*): \ \norm{y^*} \leq \kappa \norm{x^*}\}
\]
for any $\kappa > \bar\kappa_u$. Thus, in particular, the infimum of $\norm{y^*}$ over
$y^* \in D^*M((\bar y,\bar x);(0,u))(x^*)$ is bounded by $\kappa \norm{x^*}$
and, consequently, also by $\bar\kappa_u\norm{x^*}$. This justifies the statement
$\mathbf{dN}$, and $\mathbf N$ follows as a special case by choosing $u=0$,
see also \cite[Theorem 4.1]{HenJouOut02}.
The inclusion in $\mathbf{\widehat{N}}$ is obtained by polarization from the inclusion $\subset$ 
in $\mathbf T$ which is generally valid:
\begin{align*}
	\widehat N_{M(\yb)}(\xb)
	&\supset
	\bigl(DM(\yb,\xb)(0)\bigr)^\circ
	=
	\left\{x^*\,\middle|\,\skalp{x^*,u}\leq 0\,\forall u\in DM(\yb,\xb)(0)\right\}\\
	&=
	\left\{x^*\,\middle|\,\exists y^*\in\R^m\colon\,
		\skalp{(y^*,x^*),(0,u)}\leq 0\,\forall (0,u)\in T_{\gph M}(\yb,\xb)\right\}\\
	&\supset
	\left\{x^*\,\middle|\,\exists y^*\in\R^m\colon\,(y^*,x^*)\in\widehat N_{\gph M}(\yb,\xb)\right\}\\
	&=
	-\dom \widehat{D}^*M(\yb,\xb).
\end{align*}
This already completes the proof.
\end{proof}

Note that the local closedness of $M(\yb)$ around $\xb$ follows from
the local closedness of $\gph M$ around $(\yb,\xb)$.

We see similar features in play as before,
namely that tangents to image sets are given
via an image of the graphical derivative
(the primal constructions preserve the image)
while normals to image sets are characterized via the domain
of the associated coderivative
(the dual constructions flip the image to the domain).
%
%Note that this interplay is converse to the one we observed in \cref{The : IC*calc}.

\begin{remark}\label{rem:cones_to_domain}
    Let us mention that \cref{The : IC*calc,The : Ccalc}
    are in fact special cases of the calculus rules for
    tangents and normals to image sets and pre-image sets, respectively,
    since $\dom M=\varphi(\gph M)$
	where $\varphi\colon\R^m\times\R^n\to\R^m$
	is given by $\varphi(y,x):=y$, while
	$M(\yb)=\phi^{-1}(\gph M)$
	for $\phi\colon\R^n\to\R^m\times\R^n$ given by $\phi(x):=(\bar y,x)$.
	Related results for tangents and (directional) limiting normals can be found
	in \cite[Corollary 4.2 and Theorem~4.3]{Be19} and \cite[Theorems~3.1 and 3.2]{BeGfrOut19}, where the inner calmness* assumption is
	imposed on
	\[\widetilde M_1(y):=\varphi^{-1}(y) \cap \gph M = (y,M(y))\]
	and the calmness assumption is imposed on
	\[\widetilde M_2(y,z):= \{x \mv \phi(x) + (y,z) \in \gph M\} =
	M(\yb + y) - z.\]
    Since the functions $y \to y$, $(y,z) \to \yb + y$, 
    and $(y,z) \to -z$ are affine,
    they are calm, and \cref{pro : SV_calm_perturb} yields
    that the inner calmness* of $M$ implies
    the inner calmness* of $\widetilde M_1$
    and the calmness of $M$ implies that of $\widetilde M_2$.
    In fact, these assumptions are equivalent;
    the inner calmness* of $M$ follows immediately from the inner calmness* of $\widetilde M_1$,
    while the calmness of $M$ follows
    again from \cref{pro : SV_calm_perturb}
    since $M(y)=\widetilde M_2(y-\yb,0)$.
\end{remark}

Finally, let us comment on some potential extensions of our results to the setting of Asplund spaces where
the calculus of (directional) limiting normals is well-developed as well, see
\cite{LongWangYang2017,Mo06}. 
\begin{remark}\label{rem:infinite_dimensions}
Let us assume that the set-valued mapping under consideration acts between
Asplund spaces which are Banach spaces where all continuous, convex functions are generically
Fr\'{e}chet differentiable, see \cite{Mo06} for details.\\
For simplicity, we start with a review of \cref{The : Ccalc}. One can easily check
that statement $\mathbf{T}$ holds as stated in the more general situation. This, however,
cannot be used to prove a counterpart of $\widehat{\mathbf N}$ since the regular normal cone
is only a subset of the polar associated with the tangent cone in general. In reflexive spaces,
we get full polarity only w.r.t.\ the \emph{weak} tangent cone, where weak convergence 
of the directions is demanded in the definition, see \cite[Theorem~1.10, Corollary~1.11]{Mo06}.
In order to transfer $\mathbf{N}$ and $\mathbf{dN}$ to the Asplund space setting, the underlying
proof from \cite{GfrOut16} has to be adjusted slightly in order to handle 
that (directional) limiting normals are weak* limits of regular normals in the 
infinite-dimensional situation, see \cite[Section~2]{LongWangYang2017} and \cite[Section~3]{Mo06} 
for details. \\
Let us now commend on \cref{The : IC*calc}. Here, the inclusion $\supset$ in statement $\mathbf{T}$ 
remains valid in the more general setting. The converse
inclusions, however, cannot be shown in the presented way since bounded sequences in infinite-dimensional
spaces do not possess convergent subsequences in general. 
Even in the setting of reflexive Banach spaces, this issue cannot be solved since one constructs
elements of the possibly larger weak tangent cone in the proof.
A similar reasoning can be used to infer that
the proof of statement $\mathbf{dN}$ does not apply in infinite dimensions. 
It is not clear how to obtain any of the inclusions of statement $\mathbf{\widehat{N}}$ due to the
already mentioned difficulties.
On the other hand, one can easily check that statement $\mathbf{N}$
stays true when using the \emph{mixed} coderivative of the involved set-valued map, see
\cite[Definition~1.32]{Mo06} for a definition.\\
Summing up these impressions, a full generalization of the patterns observed in
\cref{The : IC*calc,The : Ccalc} 
to the infinite-dimensional situation is highly questionable for regular normals.
It might be possible to carry over parts of the analysis for regular normals to so-called
Dini--Hadamard normals which are, by definition, polar to tangents, see
\cite{Nechita2012} for details.
The situation seems to be less hopeless for tangents and directional limiting normals in reflexive spaces.
However, one has to find a way to bypass the appearance of weak tangents.
For limiting normals, an extension to Asplund spaces seems to be possible, directly.
Nevertheless, we would like to point the reader's attention to the fact that due to
certain convexification effects, limiting normals turned out to be of limited practical
use in Lebesgue and Sobolev spaces which are standard in optimal control,
see \cite{HarderWachsmuth2018,Mehlitz2019a,MehlitzWachsmuth2018} for a detailed investigation.
\end{remark}

\section{Sufficient Conditions for Calmness and Inner Calmness*} \label{Sec:SC}

In this section, we connect calmness and inner calmness* with
several other continuity and Lipschitzian properties of set-valued mappings.
We begin, however, by showing that inner calmness* in the fuzzy
sense is not only a sufficient condition for validity of the characterization of tangents 
to the domain of a set-valued mapping from \cref{The : IC*calc}, but that the two
are actually equivalent.

\begin{theorem}\label{The:Equiv_GD_IC*fuzzy}
Let $M\colon\R^m \tto \R^n$ be a set-valued mapping and fix $\yb\in\dom M$ 
such that $\gph M$ is locally closed around $\{\yb\}\times\R^n$ 
and $\dom M$ is locally closed around $\yb$.
Then $M$ is inner calm* at $\yb$ in direction $v\in\Sp$ w.r.t.\ $\dom M$ with modulus $\bar\kappa_v$ in the fuzzy sense
if and only if
\begin{equation}\label{eq:fic*_char}
	v \in T_{\dom M}(\yb) 
	\quad \Longrightarrow \quad
	\bar\kappa_v
	=
	\inf\limits_{\xb\in M(\yb)}\inf_{u \in DM(\yb,\xb)(v)} \norm{u}/\norm{v} 
	< 
	\infty.
\end{equation}
In particular, $M$ is inner calm* at $\yb$
w.r.t.\ $\dom M$ in the fuzzy sense if and only if \eqref{eq:fic*_char}
holds for every direction $v\in\Sp$.
\end{theorem}
\begin{proof}
For the proof, we set 
\[
	\hat\kappa_v:=\inf\limits_{\xb\in M(\yb)}\inf\limits_{u\in DM(\yb,\xb)(v)}\norm{u}/\norm{v}.
\]
First, assume that $M$ is inner calm* at $\yb$ in direction $v$ w.r.t.\ $\dom M$ in the fuzzy sense with
modulus $\bar\kappa_v$. 
Then \cref{The : IC*calc} yields the finiteness of $\hat\kappa_v$
as well as $\hat\kappa_v \leq \bar\kappa_v$.
Supposing that $\hat\kappa_v<\bar\kappa_v$ holds, we find some $\varepsilon>0$, $\bar x\in M(\yb)$,
and $u\in\R^n$ as well as sequences $(t_k)\downarrow 0$, $(u_k)\to u$, and $(v_k)\to v$ such that
$(\yb+t_kv_k,\xb+t_ku_k)\in\gph M$ and $\norm{u_k}\leq(\bar\kappa_v-\varepsilon)\norm{v_k}$ hold
for all $k\in\N$. Setting $y_k:=\yb+t_kv_k$ and $x_k:=\xb+t_ku_k$ for all $k\in\N$, we find
$(y_k)\to\yb$, $(x_k)\to\xb$, and $y_k\in\dom M$ as well as
$\norm{x_k-\xb}\leq(\bar\kappa_v-\varepsilon)\norm{y_k-\yb}$ for all $k\in\N$.
This contradicts the definition of the modulus of inner calmness* of $M$ at $\yb$
w.r.t.\ $\dom M$ in direction $v$ in the fuzzy sense.
Hence, we have shown the validity of \eqref{eq:fic*_char}.

Next, assume that \eqref{eq:fic*_char} holds. 
If $v \notin T_{\dom M}(\yb)$ is valid,
there is nothing to prove. Thus, let us assume $v \in T_{\dom M}(\yb)$.
Then \eqref{eq:fic*_char} yields the existence of
$\xb \in M(\yb)$ and $u\in\R^n$ together with sequences
$(t_k) \downarrow 0$, $(u_k)\to u$, and $(v_k) \to v$
such that $(\yb + t_k v_k, \xb + t_k u_k) \in \gph M$
and $\norm{u_k} \leq \kappa \norm{v_k}$ hold for all $\kappa > \hat\kappa_v$ 
and sufficiently large $k\in\N$. We set $y_k:=\yb+t_kv_k$ and $x_k:=\xb+t_ku_k$ for all $k\in\N$.
Hence, $\norm{x_k - \xb}\leq \kappa \norm{y_k - \yb}$ is valid for all $\kappa>\hat\kappa_v$ and
sufficiently large $k\in\N$
showing the inner calmness* of $M$ at $\yb$ in the fuzzy sense w.r.t.\ $\dom M$ in direction $v$ 
while the modulus satisfies $\bar\kappa_v \leq \hat\kappa_v$.
The first part of the proof now yields $\bar\kappa_v=\hat\kappa_v$.
\end{proof}

Below, we recall some prominent continuity and Lipschitzianity notions from the literature,
see e.g.\ \cite{BeGfrOut19,DoRo14,Mo06}.

\begin{definition}\label{def:other_properties_of_set_valued_maps}
Consider a set-valued mapping $M\colon\R^m \tto \R^n$ which possesses a locally
closed graph around $(\yb,\xb)\in\gph M$.
We say that
\begin{itemize}
    \item[(i)] $M$ is \emph{inner semicontinuous} at $(\yb,\xb)$ w.r.t.\ a 
    	set $\Omega \subset \R^m$ if for every sequence $(y_k)\subset\Omega$ satisfying
    	$(y_k)\to\yb$, there exists  a sequence $(x_k) \to \xb$ such that $x_k \in M(y_k)$ 
    	holds for sufficiently large $k\in\N$. In case $\Omega:=\R^m$, we simply say that $M$
    	is inner semicontinuous at $(\yb,\xb)$.
    \item[(ii)] $M$ is {\em inner calm} at $(\yb,\xb)$ w.r.t.\ a set $\Omega\subset\R^m$
    	if there exists $\kappa>0$
    	such that for every sequence $(y_k)\subset\Omega$ satisfying $(y_k)\to\yb$,
    	there exists a sequence $(x_k)\to\xb$
    	such that $x_k\in M(y_k)$ and 
    	\[
    		\norm{x_k - \xb} \leq \kappa \norm{y_k - \yb}
    	\]
    	hold for sufficiently large $k\in\N$. In case $\Omega:=\R^m$, we simply say that $M$
    	is inner calm at $(\yb,\xb)$.
    \item[(iii)] $M$ has the {\em isolated calmness} property at $(\yb,\xb)$
    	if there exist $\kappa > 0$ and neighborhoods $U$ of $\xb$ and $V$ of $\yb$ such that
    	the following estimate holds:
    	\begin{equation}\label{eq:isolated_calmness}
        	M(y)\cap U \subset \xb + \kappa \norm{y - \yb}\B
        	\quad \forall y \in V.
	    \end{equation}
    \item[(iv)] $M$ has the {\em Aubin property} at $(\yb,\xb)$
   	 if there exist $\kappa > 0$ and neighborhoods $U$ of $\xb$ and $V$ of $\yb$ such that
   	 the following estimate holds:
    	\begin{equation*}
        	M(y^{\prime})\cap U \subset M(y) + \kappa \norm{y-y'}\B
        	\quad \forall y,y'\in V.
    	\end{equation*}
\end{itemize}
The modulus of inner calmness (w.r.t.\ $\Omega$), isolated calmness, and the Aubin property, respectively, 
is defined as the infimum over all $\kappa$ satisfying the respective Lipschitz estimate
from above.
\end{definition}

Let $M\colon\R^m\tto\R^n$ be a set-valued mapping with locally 
closed graph around some point $(\yb,\xb)\in\gph M$.
Inner semicontinuity and inner calmness at $(\yb,\xb)$ are clearly stronger than
inner semicompactness and inner calmness* at $\yb$, respectively.
Thus, in \cref{The : IC*calc}, we focused on inner semicompactness, 
and inner calmness* (in the fuzzy sense)
in order to have the results more general,
taking into account that it is very easy to derive the estimates
based on inner semicontinuity and inner calmness -
one just has the fixed $\xb$ instead of the unions.
More importantly, it seems that these weaker notions have
a better chance to be satisfied in practically relevant settings,
see \cite[Section 3]{Be19}.

Let us note that $M$ is inner calm at $(\yb,\xb)\in\gph M$
if and only if there are
a constant $\kappa>0$ as well as a neighborhood $V$ of $\yb$ such that
\[
	\xb\in M(y)+\kappa\norm{y-\yb}\B\quad\forall y\in V
\]
holds. Thus, whenever $M$ possesses the Aubin property at $(\yb,\xb)$, it is also inner
calm there. However, we note that the modulus of inner calmness might be strictly smaller
than the modulus of the Aubin property in this situation.
It is both obvious and well known that whenever $M$ is isolatedly calm at $(\yb,\xb)$
or possesses the Aubin property at this point, then $M$ is also calm
there. More precisely, isolated calmness
implies calmness with the same modulus, while the Aubin property
implies calmness with not larger modulus. 

Observe that (isolated) calmness and inner calmness of a set-valued 
mapping are not related to each other. On the one hand,
the single-valued function $m\colon\R\to\R$,
given by $m(y):=1/y$ if $y\in\R\setminus\{0\}$ and $m(0):=0$,
is isolatedly calm at $(0,0)\in\gph m$, and hence also calm,
but it is not inner calm there.
On the other hand, $M\colon\R\tto\R$ given by $M(y):=[-|y|^{1/2},\infty)$ for
all $y\in\R$ is inner calm but not calm at $(0,0)$,
let alone isolatedly calm.
The following lemma, however, shows that under additional assumptions, 
isolated calmness may imply inner calmness or inner calmness* (in the fuzzy sense).

\begin{lemma}\label{Lem:ICtoIC}
	Let $M\colon\R^m\tto\R^n$ be a set-valued mapping whose graph
	is locally closed around $\{\yb\}\times\R^n$ for some $\yb\in\dom M$. 
	Then the following assertions hold.
	\begin{itemize}
    	\item[(i)] Fix a point $\xb\in M(\yb)$. 
    				If $M$ is inner semicontinuous w.r.t.\ $\dom M$ and isolatedly calm at $(\yb,\xb)$,
    				then it is inner calm w.r.t.\ $\dom M$ there.
    				
    	\item[(ii)] If $M$ is inner semicompact at $\yb$ w.r.t.\ $\dom M$ 
    		and isolatedly calm at $(\yb,x)$ 
    		for each $x \in M(\yb)$, then it is inner calm* at $\yb$ w.r.t.\ $\dom M$.
	\end{itemize}
\end{lemma}
\begin{proof}
		Let us start with the proof of (i).
		Since $M$ is isolatedly calm at $(\yb,\xb)$, 
		we find a constant $\kappa>0$ and neighborhoods $U$ of $\xb$ as well as $V$ of $\yb$
		satisfying \eqref{eq:isolated_calmness}. 
		Given $(y_k)\subset\dom M$ with $(y_k)\to\yb$, 
		we find a sequence $(x_k)\to\xb$ satisfying $x_k \in M(y_k)$ for all $k\in\N$ 
		by inner semicontinuity of $M$ at $(\yb,\xb)$ w.r.t.\ $\dom M$.
		Due to \eqref{eq:isolated_calmness}, we find $\norm{x_k-\bar x}\leq\kappa\norm{y_k-\bar x}$
		for large enough $k\in\N$ showing the inner calmness of $M$ at $(\yb,\xb)$ w.r.t.\ $\dom M$.\\
		Next, we show (ii).
		First, since points $x \in M(\yb)$ are isolated, there can be only countably
		many of them, i.e., $M(\yb)=\{x^l \mv l \in N\}$ for some $N \subset \N$.
		For each $l \in N$ let $\varepsilon_l > 0$ be such that
		$M(\yb) \cap (x^l + \varepsilon_l \B) = \{x^l\}$.
		By inner semicompactness of $M$ at $\yb$ w.r.t.\ $\dom M$,
		for every $(y_k)\subset\dom M$ with $(y_k)\to\yb$, 
		we find a sequence $(x_k)$ which converges along a subsequence
		to some $\tilde x\in \R^n$ such that $x_k\in M(y_k)$ holds along this subsequence.
		By closedness of $\gph M$ around $\{\yb\}\times\R^n$, we find $\tilde x = x^l$
		for some $l \in N$, and we denote by $l(y_k)$ the smallest of these numbers, i.e., 
		\[
		l(y_k) := \min\left\{l \in N\,\middle|\, \liminf_{k \to \infty} \dist(x^l,M(y_k)) = 0\right\}.
		\]
		In particular, for each $j \in N$ with $j < l(y_k)$, we get
		$\liminf_{k \to \infty} \dist(x^j,M(y_k)) \geq \varepsilon_j/2 > 0$.
		Let us now prove that the set
		\[
		    \widetilde N := \{l \in N \mv  \exists\,(y^l_k) \subset \dom M\colon\,
		    (y^l_k) \to \yb, l = l(y^l_k)\}
		\]
		is finite.
		By contraposition, suppose that $\widetilde N$ is infinite and
		for each $l \in \widetilde N$, let $k_l$ be such that
		$\norm{y^l_{k_l} - \yb} \leq 1/l$ and
		$\dist(x^j,M(y^l_{k_l})) \geq \varepsilon_j/4 > 0$ for all $j \in N$ with $j < l$.
		Consider the sequence $(y^l_{k_l}) \subset \dom M$
		for indices $l \in \widetilde N$. By construction, $(y^l_{k_l}) \to \yb$ holds as $l\to\infty$.
		Then, for each $j \in N$, we have $\dist(x^j,M(y^l_{k_l})) \geq \varepsilon_j/4 > 0$
		for all $l \in \widetilde N$ with $l > j$ and, consequently,
		$\liminf_{l \to \infty} \dist(x^j,M(y^l_{k_l})) \geq \varepsilon_j/4 > 0$ follows.
		This, however, contradicts the assumed inner semicompactness of $M$.\\
		The rest of the proof now follows easily.
		Since $M$ is isolatedly calm at $(\yb,x^l)$ for each $l \in \widetilde N$, we find
		constants $\kappa_{l}>0$ and neighborhoods $U_{l}$ of $x^l$ as well as $V_{l}$ of
		$\yb$ such that \eqref{eq:isolated_calmness} holds with
		$\kappa:=\max_{l \in \widetilde N}\kappa_{l}$, 
		$U:=U_{l}$, and $V:=V_{l}$ for each $l \in \widetilde N$.
		Consider a sequence $(y_k)\subset\dom M$ with $(y_k)\to\yb$
		together with the corresponding point $x^{l(y_k)} \in M(\yb)$ and
		sequence $(x_k)$ such that $x_k \in M(y_k)$ and $(x_k) \to x^{l(y_k)}$
		hold along a subsequence.
		Since $l(y_k) \in \widetilde N$ holds, we obtain
		$\norm{x_k-x^{l(y_k)}}\leq \kappa \norm{y_k-\yb}$, showing that
		$M$ is inner calm* at $\yb$ w.r.t.\ $\dom M$.
\end{proof}
We would like to point out again that the actual modulus of inner calmness (inner calmness*) 
can be smaller than the (supremum of the) underlying modulus (moduli) of isolated calmness. 
Exemplary, by means of the set-valued mapping $M\colon\R\tto\R$ given by 
$M(y):=\{k(y+1)\,|\,k\in\N\}$ for each
$y\in\R$, one can easily check that the assumptions of setting (ii) hold at $\bar y:=0$. 
Particularly, $M$ is isolatedly calm at each point $(0,k)$, $k\in\N$, with modulus $k$. 
The supremum of all these moduli is, obviously, not finite. 
On the other hand, $M$ is inner calm* at $\bar y$ with modulus $1$.

The main reason why we consider isolated calmness and the Aubin property is that
these conditions can be characterized via generalized derivatives as follows.
For a set-valued mapping $M\colon\R^m\tto\R^n$ possessing locally closed
graph around $(\yb,\xb)\in\gph M$, isolated calmness of $M$ at $(\yb,\xb)$ 
is equivalent to 
\begin{equation}\label{eq:ICcrit}\tag{LRC}
	DM(\yb,\xb)(0)=\{0\},
\end{equation}
see \cite{Levy96}, which is referred to as Levy--Rockafellar criterion in the literature.
On the other hand, $M$ possesses the Aubin property at $(\yb,\xb)$ if and only if
\begin{equation}\label{eq:Mcrit}\tag{MC}
	D^*M(\yb,\xb)(0)=\{0\} 
\end{equation}
is valid, see \cite[Theorem 9.40]{RoWe98}, and the latter is known as
Mordukhovich criterion. We note that the Aubin property can be characterized 
in terms of the graphical derivative as well,
see \cite[Theorem~7.5.4]{AubinEkeland1984} and \cite[Theorem 4B.2]{DoRo14}.

Consider a single-valued continuous function $\varphi\colon \R^n \to \R^m$.
For arbitrary $\bar x\in\R^n$, it is well known that the Aubin property 
of $\varphi$ at $(\bar x,\varphi(\bar x))$
equals local Lipschitzness of $\varphi$ at $\bar x$ which is, thus, characterized
by $D^*\varphi(\xb)(0)=\{0\}$.
On the other hand, since $\varphi$
is inner semicontinuous at each point of its graph,
calmness, inner calmness, and isolated calmness coincide
and correspond to the standard definition of calmness of single-valued mappings,
see \cref{sec:stability_notions}.
Consequently, we obtain the following simple corollary,
see also \cite[Proposition 9.24]{RoWe98}.
\begin{corollary}\label{Cor:S-V_calmness_char}
Fix $\bar x\in\R^n$.
A continuous function $\varphi\colon \R^n \to \R^m$
is calm at $\xb$ if and only if $D\varphi(\xb)(0)=0$ holds.
\end{corollary}

Exemplary, consider the continuous function $\varphi\colon\R\to\R$ given by 
$\varphi(x):=x^{3/2}\,\sin(1/x)$ for all $x\in\R\setminus\{0\}$ and $\varphi(0):=0$. 
We find $D^*\varphi(0)(0)=\R$ and $D\varphi(0)(0)=\{0\}$, i.e., $\varphi$ is calm
at $0$ but not Lipschitz at $0$.

Let us now review \cref{The : IC*calc} in the light of isolated calmness and the Aubin property.
\Cref{Lem:ICtoIC} yields the following result in terms of isolated calmness.
\begin{corollary}\label{Cor : ICcalc}
    Assume that $M\colon\R^m \tto \R^n$ has locally closed graph around
	$\{\yb\}\times \R^n$ for some $\yb\in\dom M$, that $M$ is inner semicompact at $\yb$
	w.r.t.\ $\dom M$, and that
	\eqref{eq:ICcrit} holds for each $\xb \in M(\yb)$. 
	Let $v \in \R^m$ be an arbitrary direction.
	Then we have the relations
	\begin{align*}
	    T_{\dom M}(\bar y)	& =  \bigcup_{\xb \in M(\yb)} \dom DM(\yb,\xb),\\
		\widehat{N}_{\dom M}(\yb) & = 	\bigcap_{\xb \in M(\yb)} \widehat{D}^*M(\yb,\xb)(0),\\
		N_{\dom M}(\yb) & \subset \bigcup_{\xb \in M(\yb)} D^*M(\yb,\xb)(0),\\
		N_{\dom M}(\yb;v) & \subset  \bigcup_{\xb \in M(\yb)}
			\bigcup_{u \in DM(\yb,\xb)(v)} D^*M((\yb,\xb);(v,u))(0).
	\end{align*}
\end{corollary}

Let us now revisit \cref{The : IC*calc} in terms of the Aubin property whose presence is
also sufficient for validity of inner calmness.
Clearly, if $M\colon\R^m\tto\R^n$ possesses the Aubin property at some point
$(\yb,\xb)\in\gph M$ where $\gph M$ is locally closed, then, on the one hand,
we find $\yb\in\inn\dom M$, i.e., $T_{\dom M}(\yb)=\R^m$ and
$\widehat N_{\dom M}(\yb)=N_{\dom M}(\yb)=N_{\dom M}(\yb;v)=\{0\}$ hold for all
$v\in\R^m$. On the other hand,
\cref{The : IC*calc} reproduces these trivial relations for limiting normals and directional
limiting normals, taking into account that the estimates
hold without the union (i.e., with the fixed $\xb$).
Observing that the regular coderivative is a subset of the limiting one,
we also recover $\widehat N_{\dom M}(\yb)=\{0\}$.
Finally, from $N_{\dom M}(\yb)=\{0\}$, we find
$\R^m=\left(N_{\dom M}(\yb)\right)^\circ\subset	T_{\dom M}(\yb)$
where the last inclusion follows from \cite[Theorem~6.26, Exercise~6.38]{RoWe98}.

The situation is quite similar if we look at \cref{The : Ccalc}
assuming that $M$ is isolatedly calm at $(\yb,\xb)$. In this case, \eqref{eq:isolated_calmness}
guarantees that $\xb$ is an isolated point of $M(\yb)$ which shows
$T_{M(\yb)}(\xb)=\{0\}$ and $\widehat{N}_{M(\yb)}(\xb)=N_{M(\yb)}(\xb)=N_{M(\yb)}(\xb;0)=\R^n$.
Using the Levy--Rockafellar criterion \eqref{eq:ICcrit}, \cref{The : Ccalc} recovers this
observation for tangents. This yields $\widehat{N}_{M(\yb)}(\xb)=\R^n$ and, thus,
$N_{M(\yb)}(\xb)=N_{M(\yb)}(\xb;0)=\R^n$ (note that $0$ is the only vector in $T_{M(\yb)}(\xb)$).
Thus, \cref{The : Ccalc} implies $\dom D^*M(\yb,\xb)=\dom D^*M((\yb,\xb);(0,0))=\R^n$.
Observe that we do not obtain any information about the domain of the regular coderivative
$\widehat{D}^*M(\yb,\xb)$. Exemplary, take a look at the mappings $M_1,M_2\colon\R\tto\R$ given by
$M_1(y):=[-|y|,|y|]$ and $M_2(y):=[-y^2,y^2]$ for all $y\in\R$. Clearly, we have
$DM_i(0,0)(0)=\{0\}$ for $i=1,2$, i.e., \eqref{eq:ICcrit} holds. On the other hand, we have
$\dom \widehat{D}^*M_1(0,0)=\{0\}$ and $\dom\widehat{D}^*M_2(0,0)=\R$. Thus, \eqref{eq:ICcrit}
does not yield a sharp estimate for regular normals.
Another way to look at this situation is that due to the fact that isolated calmness
of $M$ at $(\yb,\xb)$ yields the isolatedness of $\xb$ in $M(\yb)$ and, thus, $T_{M(\yb)}(\xb)=\{0\}$,
\cref{The : Ccalc} shows that \eqref{eq:ICcrit} is indeed necessary for isolated calmness.

Using the Aubin property in \cref{The : Ccalc} yields the following estimates.
\begin{corollary}\label{Cor : APcalc}
    Assume that $M\colon\R^m\tto\R^n$ possesses a locally closed graph around
    $(\bar y,\bar x)\in\gph M$ and that \eqref{eq:Mcrit} is valid.
    Let $u \in \R^n$ be an arbitrary direction.
    Then the following relations hold:
    \begin{align*}
	    T_{M(\bar y)}(\bar x) & =  DM(\bar y,\bar x)(0),\\
		\widehat{N}_{M(\bar y)}(\bar x) & \supset  -\dom \widehat{D}^*M(\bar y,\bar x),\\
		N_{M(\bar y)}(\bar x) & \subset  -\dom D^*M(\bar y,\bar x),\\
		N_{M(\bar y)}(\bar x;u) & \subset  -\dom D^*M((\bar y,\bar x);(0,u)).
	\end{align*}
\end{corollary}

Finally, we consider the \emph{first-order sufficient condition for calmness} (FOSCclm for short) given by
\begin{equation}\label{eq:FOSCclm}\tag{FOSCclm}
	D^*M((\yb,\xb);(0,u))(0)=\{0\}\quad \forall u \in DM(\yb,\xb)(0)\setminus\{0\}.
\end{equation}
As mentioned in \cref{sec:introduction}, this condition is nothing else but the calmness counterpart of
Gfrerer's FOSCMS.
The following lemma and its proof justify this fact.
\begin{lemma}\label{lem:FOSCclm}
    Let $M\colon\R^m\tto\R^n$ be a set-valued mapping with locally closed graph around 
    $(\yb,\xb)\in\gph M$. Furthermore, let \eqref{eq:FOSCclm} be valid at $(\yb,\xb)$.
    Then $M$ is calm at $(\yb,\xb)$.
\end{lemma}
\begin{proof}
    Recall that $M$ is calm at $(\yb,\xb)$ if and only if the inverse mapping $M^{-1}$
    is metrically subregular at $(\xb,\yb)$.
    Furthermore, due to \cite[Section~2.2]{GfrOut16}, the condition
    \begin{equation}\label{eq:FOSCMS}
    	0\in D^*M^{-1}((\xb,\yb);(u,0))(y^*) \ \Longrightarrow \ y^*=0
    	\quad\forall\, u \neq 0\colon\, \ 0 \in DM^{-1}(\xb,\yb)(u)
    \end{equation}
    is sufficient for metric subregularity of $M^{-1}$ at $(\xb,\yb)$.
    By the definition of the inverse mapping and \cref{lem:change_of_coordinates},
    we find
    $u \in DM(\yb,\xb)(0) 
     \ \Longleftrightarrow \
     0 \in DM^{-1}(\xb,\yb)(u)$
    for all $u\in\R^n$ and
    \begin{equation*}
        -y^* \in D^*M((\yb,\xb);(0,u))(0) 
        \ \Longleftrightarrow \
        0 \in D^*M^{-1}((\xb,\yb);(u,0))(y^*)
    \end{equation*}
    for all $y^*\in\R^m$, showing that \eqref{eq:FOSCclm} and
    \eqref{eq:FOSCMS} are equivalent and the proof is done.
\end{proof}
\noindent
Naturally, validity of \eqref{eq:FOSCclm} just for some $u$
implies calmness in direction $u$.

We point out that \eqref{eq:FOSCclm} is a refinement of \eqref{eq:ICcrit} and \eqref{eq:Mcrit}, i.e.,
it is implied by each of them.
It turns out that \eqref{eq:FOSCclm} also yields all the calculus estimates
and yet it is not too strong to make some of them irrelevant.
\begin{theorem}\label{thm:estimates_via_FOSCclm}
	Let $M\colon\R^m\tto\R^n$ be a given set-valued mapping and fix $\yb\in\dom M$.
	Then the following assertions hold.
	\begin{enumerate}
		\item[(i)] Assume that $M$ has locally closed graph around $\{\yb\}\times \R^n$,
			that $M$ is inner semicompact at $\yb$ w.r.t.\ $\dom M$, 
			and that \eqref{eq:FOSCclm} holds at $(\yb,\xb)$
			for each $\xb\in M(\yb)$. Then the relations from \cref{Cor : ICcalc} are valid.
		\item[(ii)] Fix $\xb\in M(\yb)$, let $M$ possess a locally closed graph around $(\yb,\xb)$,
			and assume that \eqref{eq:FOSCclm} holds at $(\yb,\xb)$.
    		Then the relations from \cref{Cor : APcalc} are valid.
   	\end{enumerate}
\end{theorem}
\begin{proof}
    Taking \cref{lem:FOSCclm} into account, assertion (ii) is obvious from \cref{The : Ccalc}.
    Thus, we only need to prove assertion (i).
    
    The estimate for the limiting normals holds true due to the assumed inner semicompactness
    of $M$ at $\yb$ w.r.t.\ $\dom M$, see \cref{The : IC*calc}.
    Similarly, inner semicompactness of $M$ at $\yb$ w.r.t.\ $\dom M$
    readily yields the more complicated estimate
    for the directional limiting normals presented in \cref{The : IC*calc}.
    However, \eqref{eq:FOSCclm} guarantees that
    \[
    	\bigcup_{u \in DM(\yb,\xb)(0) \cap \Sp} D^*M((\yb,\xb);(0,u))(0)
    \]
    reduces to $\{0\}$ for each $\xb\in M(\yb)$ which is why the simpler estimate from
    \cref{Cor : ICcalc} holds.
    
    Notably, this actually implies the other two relations for tangents and regular normals. 
    Indeed, given $v \in T_{\dom M}(\yb)$,
    in particular, we have $0\in N_{\dom M}(\yb;v) \neq \varnothing$ by definition of
    the directional limiting normal cone.
    Hence, the estimate for directional normals yields the existence of $\xb \in M(\yb)$
    together with $u \in DM(\yb,\xb)(v)$, i.e., $v \in \dom DM(\yb,\xb)$.
    Since the opposite inclusion is always valid, see \cref{The : IC*calc}, we are done.
    Finally, the estimate for regular normals follows from polarization, see the proof
    of \cref{The : IC*calc} as well.
\end{proof}

The question remains whether \eqref{eq:FOSCclm} only implies the estimates of \cref{Cor : ICcalc}
or also inner calmness* (in the fuzzy sense).
Thanks to the equivalence from \cref{The:Equiv_GD_IC*fuzzy}, we 
can indeed provide a fuzzy version of \cref{Lem:ICtoIC} in terms of \eqref{eq:FOSCclm}.
In order to handle the first statement of \cref{Lem:ICtoIC} as well,
we formally introduce the naturally arising concept of
\emph{inner calmness in the fuzzy sense} for the sake of completeness.
Expectedly, we say that a set-valued mapping $M\colon\R^m\tto\R^n$ 
is inner calm at $(\yb,\xb)\in\gph M$ w.r.t.\
$\dom M$ in the fuzzy sense if $\dom M$ is locally closed at $\yb$ while 
for each direction $v\in T_{\dom M}(\yb)\cap\Sp$, we find
a constant $\kappa_v>0$, a sequence $(y_k)\subset\dom M$ converging to $\yb$ from $v$,
and a sequence $(x_k)$ satisfying $x_k\in M(\yb)$ as well as 
$\norm{x_k-\bar x}\leq\kappa_v\norm{y_k-\yb}$ for sufficiently large $k\in\N$.
This notion follows from the definition of inner calmness* in the fuzzy sense, 
see \cref{def:inner_calmness*_in_fuzzy_sense}, by fixing $\xb\in M(\yb)$. It is, thus, not
surprising that one can state a similar result as in \cref{The:Equiv_GD_IC*fuzzy} in order to
characterize inner calmness in the fuzzy sense with the aid of the graphical derivative of
$M$ at $(\yb,\xb)$.

\begin{corollary}\label{cor:IC*_via_FOSCclm}
	Let $M\colon\R^m\tto\R^n$ be a set-valued mapping and fix $\yb\in\dom M$ such that $\gph M$
	is locally closed around $\{\yb\}\times\R^n$. 
	Then the following assertions hold.
	\begin{enumerate}
		\item[(i)] 	Fix a point $(\yb,\xb)\in\gph M$. If $M$ is inner semicontinuous at
				$(\yb,\xb)$ w.r.t.\ $\dom M$ while \eqref{eq:FOSCclm} holds at this point, then
				$M$ is inner calm in the fuzzy sense at $(\yb,\xb)$ w.r.t.\ $\dom M$.
		\item[(ii)]	Let $M$ be inner semicompact at $\yb$ w.r.t.\ $\dom M$
				and let \eqref{eq:FOSCclm} be valid
				at $(\yb,\xb)$ for each $\xb\in M(\yb)$.
				Then $M$ is inner calm* at $\yb$ w.r.t.\ $\dom M$ in the fuzzy sense.
	\end{enumerate}
\end{corollary}
\begin{proof}
	Note that by inner semicompactness of $M$ at $\yb$ w.r.t.\ $\dom M$, $\dom M$ is locally closed
	at $\yb$ in both situations.\\
	Let us start with the proof of (i).
	Inspecting \cref{The:Equiv_GD_IC*fuzzy}, it is easy to
	show that $M$ is inner calm at $(\yb,\xb)$ w.r.t.\ $\dom M$ in the fuzzy sense 
	if and only if
	\[
		v\in T_{\dom M}(\yb)\cap\Sp
		\quad \Longrightarrow \quad
		\inf_{u\in DM(\yb,\xb)(v)}\norm{u}/\norm{v}<\infty
	\]
	is valid. 
	Similar arguments as in the proof of \cref{thm:estimates_via_FOSCclm}
	can be used in order to infer that $T_{\dom M}(\yb)=\dom DM(\yb,\xb)$ holds
	due to the postulated inner semicontinuity of $M$ at $(\yb,\xb)$ w.r.t.\ $\dom M$ and 
	validity of \eqref{eq:FOSCclm} at this point. 
	Thus, the statement follows.\\		
	For the proof of (ii), due to \cref{The:Equiv_GD_IC*fuzzy} and assertion (i) 
	of \cref{thm:estimates_via_FOSCclm},
	it remains to show
	\[
		v\in\bigcup\limits_{\xb\in M(\yb)}\dom DM(\yb,\xb)\cap\Sp
		\ \Longrightarrow \
		\inf\limits_{\xb\in M(\yb)}\inf\limits_{u\in DM(\yb,\xb)(v)}\norm{u}/\norm{v}<\infty.
	\]
	The latter, however, is obvious. 
\end{proof}

In \cref{fig:calmness_conditions}, we summarize the findings of this section regarding the
relationship between all the introduced calmness-type conditions.

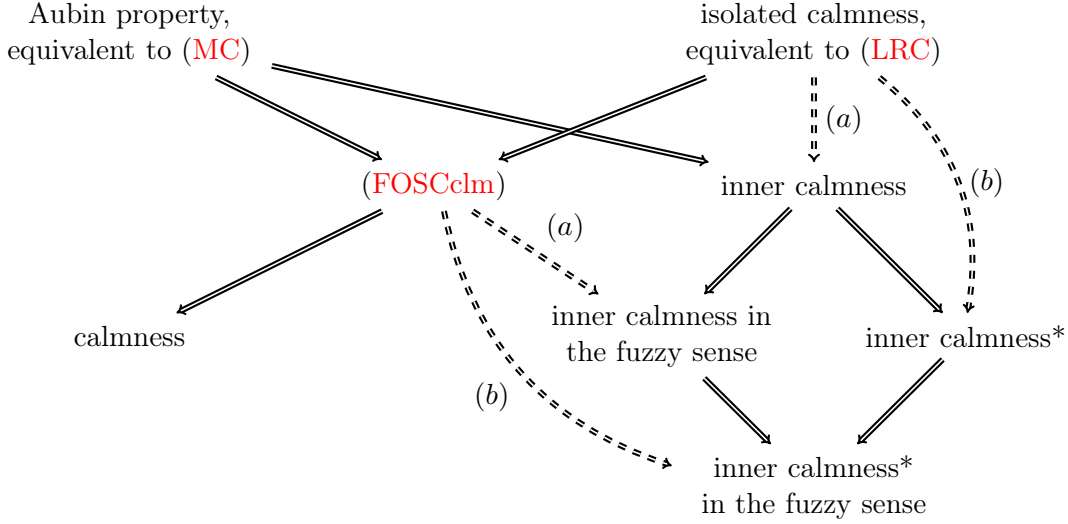
\begin{figure}[h]
\centering
\begin{tikzpicture}[->]

  \node[punkt] at (-1,0) 	(A){Aubin property, equivalent to \eqref{eq:Mcrit}};
  \node[punkt] at (8,0) 	(B){isolated calmness, equivalent to \eqref{eq:ICcrit}};
  \node[punkt] at (3,-2) 	(C){\eqref{eq:FOSCclm}};
  \node[punkt] at (8,-2)    (D){inner calmness};
  \node[punkt] at (-1,-4) 	(E){calmness};
  \node[punkt] at (10,-4) 	(F){inner calmness*};
  \node[punkt] at (8,-6)	(G){inner calmness* in the fuzzy sense};
  \node[punkt] at (6,-4)	(H){inner calmness in the fuzzy sense};

  \path     (A) edge[-implies,thick,double] node {}(C)
            (A) edge[-implies,thick,double] node {}(D)
            (B) edge[-implies,thick,double] node {}(C)
            (B) edge[-implies,thick,double,dashed] node[right] {$(a)$}(D)
            (B) edge[-implies,thick,double,dashed,bend left] node[right]{$(b)$}(F)
            (C) edge[-implies,thick,double] node {}(E)
            (C) edge[-implies,thick,double,dashed] node[above right] {$(a)$}(H)
            (C) edge[-implies,thick,double,dashed,bend right] node[below left] {$(b)$}(G)
            (D) edge[-implies,thick,double] node {}(F)
            (D) edge[-implies,thick,double] node {}(H)
            (F) edge[-implies,thick,double] node {}(G)
            (H) edge[-implies,thick,double] node {}(G);
            
\end{tikzpicture}
\caption{
	Relations between the calmness-type conditions for mappings with a closed graph.
	Dashed relations only hold under additional assumptions in general.
	\newline
	The additional assumptions $(a)$ and $(b)$
	stand for inner semicontinuity and inner semicompactness, respectively, see \cref{Lem:ICtoIC,cor:IC*_via_FOSCclm}.
}
\label{fig:calmness_conditions}
\end{figure}
It remains an open question
whether the validity of \eqref{eq:FOSCclm} together with inner semicompactness of the
underlying mapping is enough to already yield inner calmness*.
At the moment, we are not aware of a counterexample.
\if{
Indeed, this does not
hold for arbitrary mappings. Exemplary, consider
$M\colon\R\tto\R$ given by $M(y):=[0,\infty)$ for all $y\in\R\setminus\{1/k\,|\,k\in\N\}$
and $M(1/k):=\{k^{-1/2}\}$ for all $k\in\N$. Observe that $\gph M$ is not closed
in a neighborhood of $\{0\}\times\R$. However, extending the definition of the variational
objects to locally non-closed sets, for all $\xb\in M(0)$, \eqref{eq:FOSCclm}
is valid at $(0,\xb)$ and $M$ is inner semicompact at $\yb:=0$. On the other hand, $M$ is not
inner calm* at $\yb$. We are not aware of a counterexample for mappings with a closed graph.}\fi

\section{Some Calculus Rules}\label{sec:recovering_calculus}

In this section, we show how the two theorems in \cref{Sec:Main}
translate into standard calculus rules.
For brevity, in the first part, where we address
standard calculus rules related to elementary set operations, we provide
only outlines with essential information,
and we forgo presenting actual formulas and other details.
Later, when we deal with generalized derivatives of marginal
functions as well as chain and product rules for the generalized
differentiation of set-valued mappings, we present a detailed exposition.

\subsection{Calculus for Sets}\label{sec:elementary_calculus}

We consider four prototypical types of sets arising from elementary transformations,
whose tangents and normals are to be computed:
sums and intersections of closed sets as well as
images and pre-images of closed sets under continuous transformations.
Here, the main purpose is to show the connection to the theorems in
\cref{Sec:Main} as well as to emphasize the structural similarities within
these four prototypes - see e.g.\ the respective sufficient conditions.
Note, however, that for the sum rule and the image rule,
based on \Cref{The : IC*calc}, the estimates for tangents
and regular normals, as well as the corresponding
sufficient conditions are new
(tangents to the image sets appeared in \cite[Theorem~4.1, Corollary~4.2]{Be19}).
More precisely, we get equalities under fuzzy inner calmness*,
which seems less restrictive than the corresponding
convexity-based assumptions in \cite{RoWe98},
see also the comments after \cite[Corollary~4.2]{Be19}.

In each case, the desired formulas can be derived by computing
derivatives of a problem-tailored associated set-valued mapping.
We always hint how this can be done by means of the two
elementary results from \cref{lem:change_of_coordinates,lem:cartesian_products}.
On top of it,
we also comment on the closedness assumptions of the relevant theorem from \cref{Sec:Main},
refer to the analogous estimates in \cite{RoWe98},
and, finally, provide sufficient conditions for the crucial calmness-type assumption.

{
\setlength{\belowdisplayskip}{3pt}
\setlength{\belowdisplayshortskip}{3pt}
\setlength{\abovedisplayskip}{6pt}
\setlength{\abovedisplayshortskip}{6pt}

\subsubsection*{Sum Rule}
First, we consider the sum rule based on \Cref{The : IC*calc}.
\begin{description}[leftmargin=1em]
  \setlength\itemsep{0em}
    \item[input:] closed sets $D_1,\ldots,D_\ell \subset \R^m$,
        $D:=D_1+\ldots+D_\ell$;
    \item[associated mapping:] $M_1\colon \R^m \tto (\R^m)^\ell$, given by
        \[
	    M_1(y)
	    := 
	    \{(y_1,\ldots,y_\ell) \in D_1 \times \ldots \times D_\ell \mv y_1 + \ldots + y_\ell = y\},
        \]
        satisfies $\dom M_1 = D$ and
        \[
        \gph M_1
	    =
	    \{
		(y,(y_1,\ldots,y_\ell))
		\,|\,
		(y_1,\ldots,y_\ell,y-y_1-\ldots-y_\ell)\in D_1\times\ldots\times D_\ell\times\{0\}
	    \};
	    \]
    \item[closedness:] $\gph M_1$ closed by closedness of $D_1,\ldots,D_\ell$,
        $\dom M_1$ closed by assumption or by
        the inner semicompactness of $M_1$ 
        (e.g.\ if all except at most one of the sets $D_1,\ldots,D_\ell$
        are bounded and, thus, compact);
    \item[derivatives of $M_1$:] use \cref{lem:change_of_coordinates,lem:cartesian_products};
    \item[estimates:] similar to \cite[Exercise~6.44]{RoWe98},
        estimates for tangents can be enriched
        using bounds in terms of the modulus of fuzzy inner calmness*;
    \item[sufficient conditions:] isolated calmness of $M_1$ at
        $(\bar y, (\bar y_1,\ldots,\bar y_\ell)) \in \gph M_1$ (implied by):
        \begin{equation} \label{eq:Isol_Calmness_sum_rule}
        v_1+\ldots+v_\ell=0,\,v_i\in T_{D_i}(\bar y_i)\ i=1,\ldots,\ell
	    \quad
	    \Longrightarrow
	    \quad
	    v_1=\ldots=v_\ell=0;
        \end{equation}
        FOSCclm for $M_1$ at
        $(\bar y, (\bar y_1,\ldots,\bar y_\ell))$ (implied by):
        \[
        	\left.
        		\begin{aligned}
        			&v_1+\ldots+v_\ell=0,\,v_i\in T_{D_i}(\bar y_i)\, i=1,\ldots,\ell,\\
        			&(v_1,\ldots,v_\ell)\neq(0,\ldots,0)
        		\end{aligned}
        	\right\}
	    \quad
	    \Longrightarrow
	    \quad
	    N_{D_1}(\yb_1;v_1) \cap \ldots \cap N_{D_\ell}(\yb_\ell;v_\ell) = \{0\}\]
\end{description}

\subsubsection*{Intersection Rule}
Second, we consider the intersection rule based on \Cref{The : Ccalc}.
\begin{description}
  \setlength\itemsep{0em}
    \item[input:] closed sets $C_1,\ldots,C_\ell \subset \R^n$,
        $C:=\bigcap_{i=1}^\ell C_i$;
    \item[associated mapping:] $M_2\colon (\R^{n})^\ell \tto \R^n$, given by
        \[
	        M_2(x_1, \ldots, x_\ell) 
	        := 
            \bigcap\limits_{i=1}^\ell (C_i - x_i)
        	=
	        \{x \mv (x,\ldots,x)+(x_1,\ldots,x_\ell) \in C_1\times\ldots\times C_\ell\},
        \]
        satisfies $M_2(0,\ldots,0)=C$ and
        \[
	    \gph M_2
	    =
	    \{((x_1, \ldots, x_\ell),x) \,\mid\,
		(x,\ldots,x) + (x_1, \ldots, x_\ell) \in C_1 \times \ldots \times C_\ell
	    \};
        \]
    \item[closedness:] $\gph M_2$ closed by closedness of $C_1,\ldots,C_\ell$;
    \item[derivatives of $M_2$:] use \cref{lem:change_of_coordinates,lem:cartesian_products};
    \item[estimates:] similar to \cite[Theorem~6.42]{RoWe98},
        estimates for (directional) limiting normals can be enriched
        using bounds in terms of the (directional) calmness modulus;
    \item[sufficient conditions:] Aubin property of $M_2$ at
        $((0,\ldots,0),\bar x) \in \gph M_2$:
        \begin{equation}\label{eq:Aubin_Prop_cap_rule}
        x^*_1+\ldots+x^*_\ell=0,\,x^*_i\in N_{C_i}(\bar x)\ i=1,\ldots,\ell
	    \quad
	    \Longrightarrow
	    \quad
	    x^*_1=\ldots=x^*_\ell=0;
        \end{equation}
        FOSCclm for $M_2$ at
        $((0,\ldots,0),\bar x)$ (implied by):
        \[
        x^*_1+\ldots+x^*_\ell=0,\,x^*_i\in N_{C_i}(\bar x;u)\ i=1,\ldots,\ell,\,u\neq 0
	    \quad
	    \Longrightarrow
	    \quad
	    x^*_1=\ldots=x^*_\ell=0
	    \]
\end{description}

\subsubsection*{Image Rule}
Third, we consider the image rule based on \Cref{The : IC*calc}.
\begin{description}
  \setlength\itemsep{0em}
    \item[input:] continuous mapping $g\colon\R^n\to\R^m$, closed set
        $C\subset\R^n$, $D:=g(C)$;
    \item[associated mapping:] $M_3\colon\R^m \tto \R^n$, given by
        \[
	    M_3(y) := g^{-1}(y) \cap C,
        \]
        satisfies $\dom M_3 = D$ and $\gph M_3=\gph g^{-1} \cap (\R^m \times C)$;
    \item[closedness:] $\gph M_3$ closed by closedness of $C$
        and continuity of $g$,
        $\dom M_3$ closed (locally around $\yb$) by assumption or by
        the inner semicompactness of $M_3$
        (e.g.\ if $g^{-1}(V) \cap C$ is bounded
        for some neighborhood $V$ of $\yb$);
    \item[derivatives of $M_3$:] use the above intersection rule,
        justified e.g.\ if $g$ is locally Lipschitzian at $\xb \in M_3(\yb)$,
        since \eqref{eq:Aubin_Prop_cap_rule} reads as
        $D^*g(\xb)(0) \cap \bigl(-N_C(\xb)\bigr)=\{0\}$;
    \item[estimates:] similar to \cite[Theorem~6.43]{RoWe98},
        estimates for tangents can be enriched
        using bounds in terms of the modulus of fuzzy inner calmness*;
    \item[sufficient conditions:] isolated calmness of $M_3$ at
        $(g(\xb),\xb) \in \gph M_3$ (implied by):
        \[0\in Dg(\xb)(u), \ u \in T_C(\xb)
        \quad \Longrightarrow \quad u = 0;\]
        FOSCclm for $M_3$ at $(g(\xb),\xb)$ (implied by):
        \[D^*g(\xb;(u,0))(y^*) \cap \bigl(-N_C(\xb;u)\bigr) \neq \varnothing,\,u\neq 0
        \quad \Longrightarrow \quad y^* = 0\]
\end{description}

\subsubsection*{Pre-Image Rule}
Fourth, we consider the pre-image rule based on \Cref{The : Ccalc}.
\begin{description}
  \setlength\itemsep{0em}
    \item[input:] continuous mapping $g\colon\R^n\to\R^m$, closed set
        $D\subset\R^m$, $C:=g^{-1}(D)$;
    \item[associated mapping:] $M_4\colon\R^m \tto \R^n$, given by
        \[
	    M_4(y) := \{x \mv g(x) + y \in D\},
        \]
        satisfies $M_4(0)=C$ and $\gph M_4 = \gph (-g)^{-1} + (D\times\{0\})$;
    \item[closedness:] $\gph M_4$ closed by closedness of $D$
        and continuity of $g$;
    \item[derivatives of $M_4$:] use the above sum rule,
        justified e.g.\ if $g$ is calm at $\bar x$ where $\xb \in M_4(0)$,
        since \eqref{eq:Isol_Calmness_sum_rule} reads as
        $Dg(\xb)(0) \cap T_D(g(\xb))=\{0\}$;
    \item[estimates:] similar to \cite[Theorem~6.14, Theorem~6.31]{RoWe98},
        estimates for (directional) limiting normals can be enriched
        using bounds in terms of the (directional) calmness modulus;
    \item[sufficient conditions:] Aubin property of $M_4$ at
        $(0,\bar x) \in \gph M_4$:
        \[
        0\in D^*g(\xb)(y^*), \ y^* \in N_D(g(\xb)) \ \Longrightarrow \ y^* = 0;
        \]
        FOSCclm for $M_4$ at $(0,\bar x)$ (implied by):
        \[
        0\in D^*g(\xb;(u,v))(y^*), \ y^* \in N_D(g(\xb);v),\,u\neq 0 \ \Longrightarrow \ y^* = 0
	    \]
\end{description}
}

\subsection{Derivatives of Marginal Functions}\label{sec:marginal_functions}

On the basis of the two patterns from \cref{Sec:Main},
one could derive the majority of calculus rules
for subderivatives and subdifferentials,
such as the sum rule and the chain rule.
These two are, however, based on \cref{The : Ccalc}
and well-understood, so our approach provides little novelty there.
Hence, we only focus here on one prominent rule
related to \cref{The : IC*calc}.

To this end, we fix a lower semicontinuous, extended real-valued 
function $f\colon\R^n\times\R^m\to\overline{\R}$
and consider the associated marginal (or optimal value) function $\vartheta\colon\R^m\to\overline{\R}$
as well as the associated solution mapping $S\colon\R^m\tto\R^n$ given by
\[
	\vartheta(y):=\inf_x f(x,y)
	\qquad\qquad
	S(y):=\argmin_x f(x,y),
\]
respectively.
Let $\bar y\in\dom S$ be chosen such that $\vartheta(\bar y)$ is finite.

In \cite[Propositions~3.2 and~3.3.]{IofPen96},
the authors provide a direct proof of upper estimates for the limiting subdifferential of 
$\vartheta$ at $\bar y$ under the assumption that the solution map $S$ is inner semicontinuous 
at some point $(\bar y,\bar x)\in\gph S$ or inner semicompact at $\bar y$, respectively.
In \cite[Theorem~10.13]{RoWe98}, upper estimates for the regular and the limiting
subdifferential of $\vartheta$ at $\bar y$ are derived via the image rule
while exploiting that $\epi\vartheta$ can be interpreted
as a projection of $\epi f$ under mild assumptions.
In the light of our refined image rule from \cref{sec:elementary_calculus}, one has
to impose some inner semicompactness or inner calmness* (in the fuzzy sense)
of the mapping
\begin{equation*}
 \Psi(y,\alpha) := (\R^n\times\{(y,\alpha)\})\cap\epi f
\end{equation*}
at $(\bar y,\vartheta(\bar y))$ for that purpose.

Here, we strike a slightly different path.
Let us consider the problem-tailored level set mapping $M\colon\R^m\times\R \tto \R^n$
given by
\[
	M(y,\alpha)
	:=
	\{x \in \R^n\,|\,f(x,y)\leq \alpha\}
	= 
	\{x \in \R^n\,|\, (x,y,\alpha) \in \epi f\}.
\]
By definition, we have $\dom M = \epi \vartheta$
and $\gph M = \pi^{-1}(\epi f)$,
where $\pi(y,\alpha,x)=(x,y,\alpha)$ merely permutes the variables,
so there are trivial relations between derivatives of $M$ and $f$,
see \cref{lem:change_of_coordinates}.
Since $\Psi(y,\alpha) = (M(y,\alpha),(y,\alpha))$,
by \cref{pro : SV_calm_perturb}, we see that it is the same
whether we apply the corresponding assumptions on $M$
or $\Psi$, see also \cref{rem:cones_to_domain}.
On the other hand, we have $S(y)=M(y,\vartheta(y))$.
Thus, the comparison of the assumptions imposed on $S$ and $M$
via \cref{pro : SV_calm_perturb} seems to be a more delicate issue
depending on the properties of $\vartheta$, which is
an interesting topic on its own, see \cite{KlKu15}.
A direct comparison (between $S$ and $\Psi$), however,
was performed in \cite[Proof of Theorem 4.2]{BeGfrOut19}.

\cref{The : IC*calc} yields the following estimates for the
generalized derivatives of $\vartheta$.
\begin{theorem}\label{The:ValueFunction}
	Fix $\bar y\in\dom S$ where $\vartheta$ is finite
	and $\dom S$ is locally closed.
	%Fix $\bar y\in\dom S$.
	Then the following assertions hold.
	\begin{itemize}
	\item[$\mathbf{d}$] \textup{Subderivative:}
		For each direction $v\in\R^m$ and $\kappa > 0$, we always have
		\[	
			\mathrm d\vartheta(\yb)(v) \leq \inf_{\xb \in S(\yb)}
			\left( 
				\inf_{u \in \kappa \norm{(v,\mathrm d\vartheta(\yb)(v))} \B} 
				\mathrm d f(\xb,\yb)(u,v) 
			\right),
		\]
		and the opposite inequality holds true if
		$M$ is inner calm* at $(\yb,\vartheta(\yb))$
		w.r.t.\ $\dom M$ in direction $(v,\mathrm d\vartheta(\yb)(v))$ 
		in the fuzzy sense
		with modulus smaller than $\kappa$.
	\item[$\mathbf{\widehat{\partial}}$] \textup{Regular subdifferential:}
		We always have
		\[
			y^* \in \widehat{\partial}\vartheta(\yb) \ \Longrightarrow \
			(0,y^*) \in \bigcap_{\xb \in S(\yb)} \widehat{\partial} f(\xb,\yb),
		\]
		and the opposite implication holds true if
		$M$ is inner calm* at $(\yb,\vartheta(\yb))$ w.r.t.\ $\dom M$
		in the fuzzy sense.
	\item[$\mathbf{\partial}$] \textup{Limiting subdifferential:}
		If $M$ is inner semicompact at $(\yb,\vartheta(\yb))$ w.r.t.\ $\dom M$, then we have
		\[
			y^* \in \partial \vartheta(\yb) \ \Longrightarrow \
			(0,y^*) \in \bigcup_{\xb \in S(\yb)} \partial f(\xb,\yb).
		\]
	\item[$\mathbf{\mathbf d\partial}$] \textup{Directional limiting subdifferential:}
		If $M$ is inner calm* at $(\yb,\vartheta(\yb))$ 
		in direction $(v,\mu)\in\epi \mathrm d\vartheta(\bar y)$ w.r.t.\ $\dom M$, we have
		\[
			y^* \in \partial \vartheta(\yb;(v,\mu)) 
			\ \Longrightarrow \
			(0,y^*) \in 
				\bigcup_{\xb \in S(\yb)} 
				\bigcup_{(u,v,\mu)\in\epi \mathrm df(\bar x,\bar y)} 
					\partial f((\xb,\yb);(u,v,\mu)).
		\]
	\end{itemize}
\end{theorem}
\begin{proof}
Let us start proving the relations for the subderivative.
Taking into account that tangents to the epigraph of a function
are precisely elements of the epigraph of its subderivative,
\cref{The : IC*calc} yields
\begin{equation}\label{eq:subder_main}
	\mu \geq \mathrm d\vartheta(\yb)(v) 
	\ \Longleftarrow \
	\exists \, \xb \in S(\yb),\,\exists\, u \in \kappa \norm{(v,\mu)}\B \colon\;
	\mu \geq \mathrm d f(\xb,\yb)(u,v)
\end{equation}
for arbitrary $\kappa>0$,
and the opposite implication holds true if
$M$ is inner calm* at $(\yb,\vartheta(\yb))$
w.r.t.\ $\dom M$ in direction $(v,\mu)$ in the fuzzy sense
with modulus smaller than $\kappa$.

Next, given any $\varepsilon > 0$, there exist
$\xb_{\varepsilon} \in S(\yb)$ and
$u_{\varepsilon} \in \kappa \norm{(v,\mathrm d\vartheta(\yb)(v))} \B$
such that
\[
	\mathrm d f(\xb_{\varepsilon},\yb)(u_{\varepsilon},v) 
	\ \leq \
	\inf_{\xb \in S(\yb)}
	\left( \inf_{u \in \kappa \norm{(v,\mathrm d\vartheta(\yb)(v))} \B}
		\mathrm d f(\xb,\yb)(u,v) 
	\right) 
	+
	\varepsilon.
\]
Using \eqref{eq:subder_main} with $\mu := \inf_{\xb \in S(\yb)}
\big( \inf_{u \in \kappa \norm{(v,\mathrm d\vartheta(\yb)(v))} \B}
\mathrm d f(\xb,\yb)(u,v) \big) + \varepsilon$
yields the claimed inequality as $\varepsilon\downarrow 0$.
Note that $u_{\varepsilon} \in \kappa \norm{(v,\mu)} \B$
may not be satisfied, but this is not a problem since
\eqref{eq:subder_main} holds also without the bound on $u$.

On the other hand, if $M$ is inner calm* at $(\yb,\vartheta(\yb))$ w.r.t.\ $\dom M$
in direction $(v,\mathrm d\vartheta(\yb)(v))$ in the fuzzy sense,
the forward implication in \eqref{eq:subder_main} with $\mu := \mathrm d\vartheta(\yb)(v)$
yields the estimate $\mathrm d\vartheta(\yb)(v) \geq \mathrm d f(\xb,\yb)(u,v)$
for some $\xb \in S(\yb)$ and
$u \in \kappa \norm{(v,\mathrm d\vartheta(\yb)(v))} \B$,
and we have equality in the formula for the subderivative.

The subdifferential estimates follow directly from the definitions provided in
\cref{sec:basic_notation} as well as \cref{The : IC*calc}.
\end{proof}

While statement $\mathbf{\partial}$ is a standard result
and statement $\mathbf{\mathbf d\partial}$ is an easy extension of
\cite[Theorem 4.2]{BeGfrOut19},
the estimates from statements $\mathbf{d}$ and $\mathbf{\widehat{\partial}}$,
based on fuzzy inner calmness*, are new.

Moreover, based on the results from \cref{Sec:SC},
we can exploit the following sufficient conditions in order to
guarantee the validity of the inner calmness* type assumptions which appear in \cref{The:ValueFunction}.
Using the Levy--Rockafellar criterion, the isolated calmness of $M$ at 
$((\yb,\vartheta(\yb)),\xb) \in \gph M$ reads as
\begin{equation*}
    \mathrm d f(\xb,\yb)(u,0) \leq 0 \ \Longrightarrow \ u = 0
    \quad \iff \quad
    \mathrm d f(\xb,\yb)(u,0) > 0 \
    \forall (u,0) \in \dom \mathrm d f(\xb,\yb),\,u\neq 0,
\end{equation*}
while FOSCclm for $M$ at $((\yb,\vartheta(\yb)),\xb)$ is given by
\begin{equation}\label{eq:FOSCclm_MF}
    \begin{aligned}
    \forall u\in\R^n\setminus\{0\}\colon
    \quad
    &\mathrm d f(\xb,\yb)(u,0) \leq 0 \\
    &\quad \Longrightarrow \quad
    \left\{
    \begin{array}{l}
        (0,y^*) \notin \partial f((\xb,\yb);(u,0,0)) \quad \forall\, y^* \in \R^m, \\
        (0,y^*) \in \partial^{\infty} f((\xb,\yb);(u,0,0))
        \ \Longrightarrow \ y^*=0.
    \end{array}\right.
   	\end{aligned}
\end{equation}
The latter claim deserves a justification.
Suppose \eqref{eq:FOSCclm_MF} holds and observe
\[
	\begin{aligned}
	(y^*,\lambda^*) &\in D^*M\big(((\yb,\vartheta(\yb)),\xb);((0,0),u)\big)(0)\\
	&\iff\quad
	(0,y^*,\lambda^*) \in N_{\epi f}\big(((\xb,\yb),f(\xb,\yb));((u,0),0)\big).
	\end{aligned}
\]
Hence, $\lambda^* \leq 0$ holds. If $\lambda^* < 0$ is valid,
we get $(0,-y^*/\lambda^*) \in \partial f((\xb,\yb);(u,0,0))$ which 
contradicts \eqref{eq:FOSCclm_MF}
and so $\lambda^* = 0$. Then, however, we obtain
$(0,y^*) \in \partial^{\infty} f((\xb,\yb);(u,0,0))$
and $y^*= 0$ follows. This proves FOSCclm for $M$ at $((\yb,\vartheta(\yb)),\xb)$.
The reverse implication is now clear as well.

Finally, note that the weaker condition
\begin{equation}\label{eq:simpler_CQ_marginal_functions}
    \forall u\in\R^n\setminus\{0\}\colon\quad
    \mathrm d f(\xb,\yb)(u,0) \leq 0 \ \Longrightarrow \
        \Bigl[(0,y^*) \in \partial f((\xb,\yb);(u,0,0))
        \ \Longrightarrow \ y^*=0\Bigr]
\end{equation}
can be used here the same way as FOSCclm and isolated calmness.
Indeed, one can derive from \cref{The : IC*calc}
the corresponding estimate for directional subdifferentials, valid under just
inner semicompactness of $M$ and employing directions of the form $(u,0,0)$,
and then observe that \eqref{eq:simpler_CQ_marginal_functions} also
makes the additional union superfluous.
Then the same arguments as those in
the proof of \cref{thm:estimates_via_FOSCclm},
that justified FOSCclm, apply here as well.

\subsection{Chain Rule for Set-Valued Mappings}\label{sec:chain_rule}

Let us consider set-valued mappings
$S_1\colon\R^n\tto\R^m$ and $S_2\colon\R^m\tto\R^\ell$ with closed graphs
as well as their composition $S_2\circ S_1\colon\R^n\tto\R^\ell$ given by
\[
	(S_2\circ S_1)(x):=\bigcup\limits_{y \in S_1(x)} S_2(y).
\]
For later use, we introduce $S\colon\R^n\tto\R^\ell$ as $S:=S_2\circ S_1$.
Furthermore, let $\Xi\colon\R^n\times\R^\ell\tto\R^m$ denote the standard
``intermediate'' mapping given by
\begin{equation}\label{eq:intermediate_map}
	\Xi(x,z):=S_1(x)\cap S_2^{-1}(z)=\{y\in S_1(x)\,|\,z \in S_2(y)\}.
\end{equation}

In \cite[Theorem~10.37]{RoWe98}, the authors use a combination
of the image rule and the pre-image rule in order to compute the
limiting coderivative of the composition $S$.
On the other hand, in \cite[Section~7]{IofPen96},
the authors prove the chain rule for the limiting coderivative via subdifferentials,
namely a combination of the rule for the marginal function and the sum rule.
In \cite[Theorem~3.11]{Mo18}, the author exploits the sum rule for the 
limiting coderivative in order to derive the chain rule for the latter.

As before, we provide a different approach utilizing our results from
\cref{The : IC*calc,The : Ccalc}.
To this end, we introduce the perturbation map
$M\colon \R^{m}\times\R^{\ell} \tto \R^n\times\R^\ell\times\R^m$
given by
\begin{equation}\label{eq:perturbation_map}
	M(p,q) := \{(x,z,y) \mv y + p \in S_1(x), z + q \in S_2(y)\},
\end{equation}
and we point out the relations
\begin{equation}\label{eq:chain_rule_setting}
    \gph S = \dom \Xi \qquad\qquad \gph \Xi = M(0,0).
\end{equation}
Clearly, the results from \cref{The : IC*calc,The : Ccalc}
can be used to establish a connection between the derivatives
of $S$ and $M$ (not $S_1$ and $S_2$).
Additionally, observing that
\[
	\gph M=\{((p,q),(x,z,y))\,|\,((x,y+p),(y,z+q))\in\gph S_1\times\gph S_2\}
\]
holds, we readily infer 
for arbitrary $((\xb,\zb),\yb)\in \gph \Xi$ the relations
\begin{align*}
    DM((0,0),(\xb,\zb,\yb))(0,0)
	&\subset
	\left\{(u,w,v) \,\middle|\,
		v \in DS_1(\xb,\yb)(u),
		w \in DS_2(\yb,\zb)(v)
		\right\},\\
    \widehat{D}^*M((0,0),(\xb,\zb,\yb))(-x^*,z^*,0)
	&=
	\left\{(-p^*,-z^*) \,\middle|\,
		x^* \in \widehat{D}^*S_1(\xb,\yb)(p^*),
		p^* \in \widehat{D}^*S_2(\yb,\zb)(z^*)
		\right\}
\end{align*}
from \cref{lem:change_of_coordinates,lem:cartesian_products}, 
and analogous estimates are valid for the limiting and the
directional limiting coderivate (with relations $=$ and $\subset$, respectively).
Moreover, all inclusions become equalities under
the following equivalence:
\begin{equation}\label{eq: TangProdRel}
	\begin{aligned}
    ((u,v),(v,w)) &\in T_{\Gr S_{1} \times \Gr S_{2}}((\xb,\yb),(\yb,\bar z)) \\
    &\iff \
    v \in DS_1(\xb,\yb)(u), w \in DS_2(\yb,\zb)(v).
    \end{aligned}
\end{equation}
Putting all these things together, we arrive at the following result.
\begin{theorem}\label{The:chain_rule}
	Fix $(\bar x,\bar z)\in\gph S$ where $\gph S$ is locally closed.
	Then the following assertions hold.
	\begin{itemize}
		\item[$\mathbf{D}$] \textup{Graphical derivative:}
			If $\Xi$ is inner calm* at $(\bar x,\bar z)$ w.r.t.\ $\dom\Xi$ in the fuzzy sense, then
			\begin{equation*}
				DS(\xb,\zb)(u) \ \subset \
				\bigcup\limits_{\bar y\in\Xi(\bar x,\bar z)}
				\bigl(DS_2(\yb,\zb) \circ DS_1(\xb,\yb)\bigr)(u).
			\end{equation*}
			The opposite inclusion holds true if,
			for all $\bar y\in\Xi(\bar x,\bar z)$,
			$M$ is calm at $((0,0),(\bar x,\bar z,\bar y))$
			while \eqref{eq: TangProdRel} is valid.
		\item[$\mathbf{\widehat{D}^*}$] \textup{Regular coderivative:}
			If $\Xi$ is inner calm* at $(\bar x,\bar z)$ w.r.t.\ $\dom\Xi$ in the fuzzy sense, then
			\begin{equation*}
				\widehat{D}^* S(\bar x,\bar z)(z^*)	\ \supset \
				\bigcap\limits_{\bar y\in\Xi(\bar x,\bar z)}
				\bigl(\widehat{D}^*S_1(\bar x,\bar y)\circ \widehat{D}^*S_2(\bar y,\bar z)\bigr)(z^*).
			\end{equation*}
			On the other hand, if $M$ is calm at each point
			$((0,0),(\bar x,\bar z,\bar y))$
			such that $\bar y\in\Xi(\bar x,\bar z)$, then we also have the following upper estimate 
			in terms of the limiting coderivative:
			\begin{equation*}
				\widehat{D}^* S(\bar x,\bar z)(z^*) 
				\ \subset \
				\bigcap\limits_{\bar y\in\Xi(\bar x,\bar z)}
				\bigl(D^*S_1(\bar x,\bar y)\circ D^*S_2(\bar y,\bar z)\bigr)(z^*).
			\end{equation*}
		\item[$\mathbf{D^*}$] \textup{Limiting coderivative:}
			If $\Xi$ is inner semicompact at $(\bar x,\bar z)$ w.r.t.\ $\dom\Xi$
			and if $M$ is calm at each point $((0,0),(\bar x,\bar z,\bar y))$
			such that $\bar y\in\Xi(\bar x,\bar z)$, then
			\begin{equation*}
				D^* S(\bar x,\bar z)(z^*)
				\ \subset \
				\bigcup\limits_{\bar y\in\Xi(\bar x,\bar z)}
				\bigl(D^*S_1(\bar x,\bar y)\circ D^*S_2(\bar y,\bar z)\bigr)(z^*).
			\end{equation*}
		\item[$\mathbf{dD^*}$] \textup{Directional limiting coderivative:}
			If $\Xi$ is inner calm* at $(\bar x,\bar z)$ w.r.t.\ $\dom\Xi$
			in some direction $(u,w)\in\R^n\times\R^\ell$ 
			and if $M$ is calm at each point $((0,0),(\bar x,\bar z,\bar y))$ with $\yb\in\Xi(\xb,\zb)$
			and in each direction $(u,w,v)$ with $v \in D\Xi((\bar x,\bar z),\yb)(u,w)$, then
			\begin{align*}
				&D^* S((\bar x,\bar z);(u,w))(z^*)\\
				& \qquad \ \subset \ 
				\bigcup\limits_{\bar y\in\Xi(\bar x,\bar z)}
				\bigcup\limits_{v \in D\Xi((\bar x,\bar z),\yb)(u,w)}
				\bigl(D^*S_1((\bar x,\bar y);(u,v))\circ D^*S_2((\bar y,\bar z);(v,w))\bigr)(z^*).
			\end{align*}
	\end{itemize}
\end{theorem}
\begin{proof}
The proof in fact follows easily from the relations \eqref{eq:chain_rule_setting},
\cref{The : IC*calc,The : Ccalc}, and the
above formulas for derivatives of $M$.

Let us start to prove the assertions from $\mathbf{D}$.
Therefore, fix $w\in DS(\xb,\zb)(u)$ for some arbitrary $u\in\R^n$.
Due to $\gph S=\dom\Xi$ and the fuzzy inner calmness* of $\Xi$ at $(\xb,\zb)$ w.r.t.\
$\dom\Xi$, we find $\yb\in\Xi(\xb,\zb)$ and $v\in D\Xi((\xb,\zb),\yb)(u,w)$
by \cref{The : IC*calc}. Due to $\gph\Xi=M(0,0)$, we have
$((u,w),v)\in T_{M(0,0)}((\xb,\zb),\yb)$. Now, we can apply
\cref{The : Ccalc} in order to find $(u,w,v) \in DM((0,0),(\xb,\zb,\yb))(0,0)$,
and this yields $w\in (DS_2(\yb,\zb)\circ DS_1(\xb,\yb))(u)$.
The opposite inclusion follows by the same steps
in reverse order while noting that equality holds in the formula for the graphical
derivative for $M$ due to \eqref{eq: TangProdRel}
and that the calmness of $M$ at all reference points is needed in order to
obtain equality in the first assertion of \cref{The : Ccalc}.

Next, we investigate the assertions from $\mathbf{\widehat{D}}^*$.
Fix an arbitrary $\yb\in\Xi(\xb,\zb)$ as well as some
$x^* \in (\widehat{D}^*S_1(\bar x,\bar y)\circ \widehat{D}^*S_2(\bar y,\bar z))(z^*)$
for arbitrarily chosen $z^*\in\R^\ell$, i.e.,
$(-x^*,z^*,0)\in\dom \widehat{D}^*M((0,0),(\xb,\zb,\yb))$.
Due to $\gph\Xi=M(0,0)$, \cref{The : Ccalc} applies here and yields
$(x^*,-z^*)\in\widehat{D}^*\Xi((\xb,\zb),\yb)(0)$.
Observing that this holds for all $\yb\in\Xi(\xb,\zb)$ while
$\Xi$ is assumed to be inner calm* at $(\xb,\zb)$ w.r.t.\ $\dom\Xi$
in the fuzzy sense, \cref{The : IC*calc} leads
to $(x^*,-z^*)\in\widehat{N}_{\dom\Xi}(\xb,\zb)$.
Since $\dom\Xi=\gph S$ is valid, we have shown the validity of the lower estimate
for the regular coderivative.

On the other hand, given $x^* \in \widehat{D}^* S(\bar x,\bar z)(z^*)$ for
arbitrary $z^*\in\R^\ell$,
\cref{The : IC*calc} implies 
$(x^*,-z^*) \in \widehat{D}^*\Xi((\bar x,\bar z),\yb)(0)$
for each $\bar y\in\Xi(\bar x,\bar z)$.
Replacing the regular coderivative by the limiting one
and utilizing the assumed calmness of $M$,
\cref{The : Ccalc} leads to $(-x^*,z^*,0)\in\dom D^*M((0,0),(\xb,\zb,\yb))$
and the upper estimate for the regular coderivative follows.

The remaining proofs for the assertions in $\mathbf{D}^*$ and $\mathbf{dD}^*$ are
analogous to the validation of the upper estimate for the regular coderivative.
\end{proof}

Let us point out that the upper estimate for the regular coderivative via an intersection
is generally tighter than the one for the limiting coderivative
which is given in terms of a union.
This is achieved by replacing the regular coderivative by the limiting one
only after the first calculus step instead of doing it at the beginning.

Exploiting the precise moduli of fuzzy inner calmness* and (directional) calmness in the
proof of \cref{The:chain_rule}, one obtains estimates
which comprise bounds.
\begin{corollary}
	Fix $(\bar x,\bar z)\in\gph S$ where $\gph S$ is locally closed.
	\begin{itemize}
		\item[$\mathbf D$] \textup{Graphical derivative:}
			If $\Xi$ is inner calm* at $(\bar x,\bar z)$
			w.r.t.\ $\dom\Xi$ in direction $(u,w)$ in the fuzzy sense with
			modulus smaller than $\kappa$, then
			\begin{align*}
                &w\in DS(\xb,\zb)(u)\\
				&\quad
				\ \Longrightarrow \
				\exists\yb\in\Xi(\xb,\zb)\,
				\exists v\in DS_1(\xb,\yb)(u)\cap\kappa\norm{(u,w)}\B\colon\,
				w\in DS_2(\yb,\zb)(v).
			\end{align*}
        \item[$\mathbf{D^*}$] \textup{Limiting coderivative:} 
            If $\Xi$ is inner semicompact at $(\bar x,\bar z)$ w.r.t.\ $\dom\Xi$ and if
			$M$ is calm at each point $((0,0),(\xb,\zb,\yb))$ such that $\yb\in\Xi(\xb,\zb)$
			with modulus smaller than $\kappa_{\yb}$, then
			\begin{align*}
				&x^*\in D^*S(\xb,\zb)(z^*)\\
				&\quad
				\ \Longrightarrow \
				\exists \yb\in\Xi(\xb,\zb)\,
				\exists y^*\in D^*S_2(\yb,\zb)(z^*)\cap\kappa_{\yb}\norm{(x^*,z^*)}\B\colon\,
				x^*\in D^*S_1(\xb,\yb)(y^*).
			\end{align*}
    \end{itemize}
Naturally, the upper estimate for the regular coderivative and the estimate 
for the directional limiting coderivative can also be
enriched by the (directional) calmness modulus.
\end{corollary}

In the subsequent remark, we comment on the graphical derivative of $\Xi$,
which appears in $\mathbf{dD}^*$,
as well as on the directional limiting coderivative of $\Xi$, which will be important later 
in the discussion about sufficient conditions for the calmness conditions appearing
in \cref{The:chain_rule}.
\begin{remark}\label{rem:graphical_derivative_xi}
	Fix $(\xb,\zb)\in\gph S$ and $\yb\in\Xi(\xb,\zb)$ as well as
	a pair $(u,w)\in\R^n\times\R^\ell$.
	Due to $\gph\Xi=M(0,0)$, we can
	apply \cref{The : Ccalc} and the formula
	for the graphical derivative of $M$ in order to find the upper estimate
	\[
		D\Xi((\xb,\zb),\yb)(u,w)
		\ \subset \
		DS_1(\xb,\yb)(u)\cap DS_2(\yb,\zb)^{-1}(w).
	\]
	On the other hand, $v \in DS_1(\xb,\yb)(u)\cap DS_2(\yb,\zb)^{-1}(w)$ belongs
	to $D\Xi((\xb,\zb),\yb)(u,w)$
	whenever $M$ is calm at $((0,0),(\xb,\zb,\yb))$
	while \eqref{eq: TangProdRel} holds.
	Moreover, directional calmness of $M$ at $((0,0),(\xb,\zb,\yb))$ in direction $(u,w,v)$ yields			
	\begin{align*}
        &(x^*,-z^*) \in D^*\Xi\big(((\xb,\zb),\yb);((u,w),v)\big)(0) \\
        & \qquad \qquad \Longrightarrow \
        x^* \in \bigl(D^*S_1((\bar x,\bar y);(u,v))\circ D^*S_2((\bar y,\bar z);(v,w))\bigr)(z^*).
	\end{align*}
\end{remark}

Let us now comment on sufficient conditions for the presence of the
additional calmness properties in \cref{The:chain_rule}.
With the aid of the Mordukhovich criterion
and the formulas for the derivatives of $M$, we can easily check
that $M$ possesses the Aubin property at some
point $((0,0),(\xb,\zb,\yb)) \in \gph M$
if and only if the condition
\begin{equation}\label{eq:CQ_chain_rule_dual}
	\ker D^*S_1(\xb,\yb)\cap D^*S_2(\yb,\zb)(0)=\{0\}
\end{equation}
holds. This is the essential qualification condition which appears frequently in the literature
in order to guarantee the validity of the chain rule for the limiting coderivative, see e.g.\
\cite[Theorem~3.9]{Mo18} or \cite[Theorem~10.37]{RoWe98}.
Moreover, it is implied, in particular, if
either $S_1$ is metrically regular at $(\xb,\yb)$
or $S_2$ has the Aubin property at $(\yb,\zb)$.
Finally, FOSCclm of $M$ at $((0,0),(\xb,\zb,\yb))$
holds if for all nonzero $(u,v,w)$
with $v \in DS_1(\xb,\yb)(u)$ and $w \in DS_2(\yb,\zb)(v)$ one has
\begin{equation*}
	\ker D^*S_1((\xb,\yb);(u,v))\cap D^*S_2((\yb,\zb);(v,w))(0)=\{0\}.
\end{equation*}

On the other hand, on the basis of \cref{rem:graphical_derivative_xi} and the Levy--Rockafellar criterion,
we infer that $\Xi$ is isolatedly calm at some point
$((\xb,\zb),\yb) \in \gph \Xi$ provided
\begin{equation}\label{eq:CQ_chain_rule_primal}
	DS_1(\xb,\yb)(0) \cap \ker DS_2(\yb,\zb)=\{0\}
\end{equation}
holds. Particularly, this is satisfied
if either $S_1$ is isolatedly calm at $(\xb,\yb)$
or $S_2$ is strongly metrically subregular at $(\yb,\zb)$.
Moreover, FOSCclm for $\Xi$ holds at $((\xb,\zb),\yb)$
if one has
\begin{equation*}
 D^*\Xi\big(((\xb,\zb),\yb);((0,0),v)\big)(0) = \{(0,0)\}
 \quad
 \forall v \in \bigl(DS_1(\xb,\yb)(0) \cap \ker DS_2(\yb,\zb)\bigr)\setminus\{0\}.
\end{equation*}
Naturally, this can be secured in terms of $S_1$ and $S_2$ by
\[
	x^* \in \bigl(D^*S_1((\bar x,\bar y);(0,v))\circ D^*S_2((\bar y,\bar z);(v,0))\bigr)(z^*)
	\ \Longrightarrow \ 
	x^*, z^* = 0,\]
provided $M$ is calm at $((0,0),(\xb,\zb,\yb))$ in direction $(0,0,v)$,
see \cref{rem:graphical_derivative_xi} again.

Finally, note that the standard approach to the chain rule from the literature typically asks
for calmness of the mapping 
$\widetilde M\colon \R^n\times\R^m \times \R^m\times\R^\ell\tto \R^n\times\R^\ell\times\R^m$ given by
\[
	\widetilde M(a,b,c,d) 
	:=
	\{(x,z,y) \,|\, (x+a,y+b) \in \gph S_1, \, (y+c,z+d) \in \gph S_2\}
\]
at all points $((0,0,0,0),(\xb,\zb,\yb))$ such that $\yb\in\Xi(\xb,\zb)$.
Again, \cref{pro : SV_calm_perturb,rem:cones_to_domain}
clarify that this assumption is equivalent to ours due to
$M(b,d)=\widetilde M(0,b,0,d)$ and $\widetilde M(a,b,c,d)=M(b-c,d)-(a,c,0)$.

\subsection{Product Rule for Set-Valued Mappings}\label{sec:product_rule}

Products of set-valued mappings are interesting on their own
and also as an auxiliary tool to handle any
binary operation on set-valued mappings, see \cite[Section 3.2.3]{Mo18}.
Arguably, the most common of such operations is addition.
Here, we focus on the products.
However, we also address the so-called \emph{decoupled sum rule} along the way.

First, we investigate the simple case of {\em componentwise products}, where
the set-valued mapping $S\colon\R^{n_1}\times\R^{n_2}\tto\R^{\ell_1}\times\R^{\ell_2}$, given by
\begin{equation}\label{eq:comp_prod}
	S(y_1,y_2) := \Gamma_1(y_1) \times \Gamma_2(y_2)
\end{equation}
for mappings $\Gamma_1\colon\R^{n_1} \tto\R^{\ell_1}$ and
$\Gamma_2\colon\R^{n_2} \tto\R^{\ell_2}$ with closed graphs,
is under consideration.
Since
$\gph S=\pi^{-1}(\gph\Gamma_1\times\gph\Gamma_2)$ holds,
where $\pi(y_1,y_2,z_1,z_2):=(y_1,z_1,y_2,z_2)$
merely permutes the variables, derivatives of $S$
can be expressed as products of the derivatives associated with $\Gamma_1$ and $\Gamma_2$,
see \cref{lem:change_of_coordinates,lem:cartesian_products}.
While for graphical derivatives and directional limiting coderivatives,
we generally only get inclusions,
equalities are obtained for the regular and the limiting coderivative.
Furthermore, it is possible to get equalities for graphical derivative and directional
limiting coderivative in particular situations where the use of \cref{lem:cartesian_products} 
can be avoided, e.g.\ if one of the mappings $\Gamma_1$ or
$\Gamma_2$ is single-valued and continuously differentiable.

Let us investigate the setting $\ell:=\ell_1=\ell_2$.
Composing the resulting mapping $S$ from \eqref{eq:comp_prod} 
with the single-valued function $\funcsum\colon\R^\ell\times\R^\ell \to \R^\ell$, 
given by $\funcsum(z_1,z_2):=z_1+z_2$,
results in
\[
	\Sigma(y_1,y_2):=(\funcsum\circ S)(y_1,y_2)=\Gamma_1(y_1) + \Gamma_2(y_2).
\]
Thus, the generalized derivatives of $\Sigma\colon\R^{n_1}\times\R^{n_2}\tto\R^\ell$
can be computed with the aid of the chain rule from \cref{sec:chain_rule}.
Continuous differentiability of $\funcsum$ implies
\eqref{eq: TangProdRel} as well as calmness of the associated perturbation
mapping \eqref{eq:perturbation_map} by \eqref{eq:CQ_chain_rule_dual}.
Hence, \cref{The:chain_rule} yields the {\em decoupled sum rule}
under the corresponding inner calmness* assumptions
imposed on the intermediate mapping
\[
	\Xi((y_1,y_2),z)
	=
	\{(z_1,z_2) \in \Gamma_1(y_1) \times \Gamma_2(y_2)\,|\,z_1 + z_2 = z\}.
\]
These are, however, intrinsically satisfied if one
of the mappings is single-valued and locally Lipschitzian.
In particular, we exemplary get
\begin{equation}\label{eq:Dec_Sum_SV}
	\begin{aligned}
	D\Sigma((\yb_1,\yb_2),\zb)(v_1,v_2) 
	&\subset
	D\Gamma_1(\yb_1,\zb - \gamma_2(\yb_2))(v_1) + D\gamma_2(\yb_2)(v_2),\\
    D^*\Sigma((\yb_1,\yb_2),\zb)(z^*) 
    &\subset
    D^*\Gamma_1(\yb_1,\zb - \gamma_2(\yb_2))(z^*) \times
    D^*\gamma_2(\yb_2)(z^*)
    \end{aligned}
\end{equation}
if $\Gamma_2=\gamma_2$ holds for some single-valued, locally Lipschitzian function
$\gamma_2\colon\R^{n_2}\to\R^\ell$.
Moreover, if $\gamma_2$ is even continuously differentiable, all the estimates in fact hold with
equality as well. This can be shown by \cref{lem:change_of_coordinates}
since we have
$\gph\Sigma=\{((y_1,y_2),z)\,|\, (y_1,z-\gamma_2(y_2)) \in \gph\Gamma_1\}$.

Next, we take a closer look at the setting $n:=n_1=n_2$.
We want to compose the resulting mapping $S$ from \eqref{eq:comp_prod}
with the single-valued function $\funcdoub\colon\R^n\to\R^n\times\R^n$, given by $\funcdoub(x):=(x,x)$,
in order to compute the generalized derivatives of the \emph{product mapping} 
$\Gamma\colon\R^n\tto\R^{\ell_1}\times\R^{\ell_2}$ represented by
\[
	\Gamma(x):= (S \circ \funcdoub)(x) = \Gamma_1(x)\times\Gamma_2(x).
\]
Again, we are going to exploit the chain rule from \cref{sec:chain_rule}
for these computations.
Since $\funcdoub$ is single-valued and continuously differentiable,
\eqref{eq: TangProdRel} clearly holds, and, moreover,
the associated intermediate mapping $\Xi$ from
\eqref{eq:intermediate_map} is single- or empty-valued
and, thus, trivially inner calm w.r.t.\ its domain 
at each point of its domain.
On the other hand, the associated perturbation mapping 
$M\colon\R^n\times\R^n\times\R^{\ell_1}\times\R^{\ell_2}
\tto\R^n\times\R^{\ell_1}\times\R^{\ell_2}\times\R^n\times\R^n$
from \eqref{eq:perturbation_map} takes the form
\[
	M(p_1,p_2,q_1,q_2)=\{(x,z_1,z_2,y_1,y_2)\,|\,x-y_i=p_i,\,q_i\in\Gamma_i(y_i)-z_i,\,i=1,2\}.
\]
\cref{The:chain_rule} now yields the following result.
\begin{theorem}\label{Thm:product_rule}
	Fix $(\bar x,(\bar z_1,\bar z_2))\in\gph\Gamma$.
	Then the following assertions hold.
	\begin{itemize}
		\item[$\mathbf{D}$] \textup{Graphical derivative:}
			We always have
			\begin{align*}
				D\Gamma(\bar x,(\zb_1,\zb_2))(u)
				\subset
				D\Gamma_1(\bar x,\bar z_1)(u)\times D\Gamma_2(\bar x,\bar z_2)(u),
			\end{align*}
			and the opposite inclusion holds if $M$ is calm at 
			$((0,0,0,0),(\bar x,\bar z_1,\bar z_2,\bar x,\bar x))$
			and
			\begin{equation*}\label{eq: TangProdRel2}
    			((u,w_1),(u,w_2)) 
    			\in 
    			T_{\gph\Gamma_1\times\gph\Gamma_2}((\xb,\zb_1),(\xb,\zb_2)) 
    			\ \iff \
    			w_i \in D\Gamma_i(\bar x,\bar z_i)(u),\; i=1,2.
			\end{equation*}
		\item[$\mathbf{\widehat{D}^*}$] \textup{Regular coderivative:}
			We always have
			\[
				\widehat{D}^*\Gamma(\bar x,(\zb_1,\zb_2))(z^*)
				\supset
				\widehat{D}^*\Gamma_1(\bar x,\bar z_1)(z_1^*)
				+
				\widehat{D}^*\Gamma_2(\bar x,\bar z_2)(z_2^*).
			\]
		\item[$\mathbf{D^*}$] \textup{Limiting coderivative:}
			If $M$ is calm at $((0,0,0,0),(\bar x,\bar z_1,\bar z_2,\bar x,\bar x))$,
			then
			\[
				D^*\Gamma(\bar x,(\zb_1,\zb_2))(z^*)
				\subset
				D^*\Gamma_1(\bar x,\bar z_1)(z_1^*)
				+
				D^*\Gamma_2(\bar x,\bar z_2)(z_2^*).
			\]
		\item[$\mathbf{dD^*}$] \textup{Directional limiting coderivative:}
			Fix a direction $(u,w_1,w_2)\in\R^n\times\R^{\ell_1}\times\R^{\ell_2}$.
			If $M$ is calm at $((0,0,0,0),(\bar x,\bar z_1,\bar z_2,\bar x,\bar x))$
			in direction $(u,w_1,w_2,u,u)$, then
			\begin{align*}
				&D^*\Gamma\bigl((\bar x,(\zb_1,\zb_2));(u,(w_1,w_2))\bigr)(z^*)\\
				&\qquad\subset
				D^*\Gamma_1((\bar x,\bar z_1);(u,w_1))(z_1^*)
				+
				D^*\Gamma_2((\bar x,\bar z_2);(u,w_2))(z_2^*).
			\end{align*}
	\end{itemize}
\end{theorem}

The qualification condition
\eqref{eq:CQ_chain_rule_dual}, which implies the calmness of $M$,
takes the precise form
\begin{equation}\label{eq:CQ_product}
	D^*\Gamma_1(\bar x,\bar z_1)(0)\cap\bigl(-D^*\Gamma_2(\bar x,\bar z_2)(0)\bigr)=\{0\}.
\end{equation}
This is clearly satisfied if
one of the mappings $\Gamma_i$, $i=1,2$, possesses the Aubin property at $(\bar x,\bar z_i)$.
In the following lemma, we summarize some
settings where \eqref{eq:CQ_product} is naturally valid.
\begin{lemma}\label{lem:CQ_product_naturally_given}
	Fix $(\bar x,(\bar z_1,\bar z_2))\in\gph\Gamma$.
	Then \eqref{eq:CQ_product} holds under
	each of the following conditions:
	\begin{itemize}
     \item[(i)] there are a locally Lipschitz continuous function 
            $\gamma\colon\R^n\to\R^{\ell_2}$
			as well as a closed set $\Omega\subset\R^{\ell_2}$ such that $\Gamma_2$ is given by
			$\Gamma_2(x):=\gamma(x)+\Omega$ (an analogous statements holds if $\Gamma_1$ admits
			such a representation);
	 \item[(ii)] the variables $x$ can be decomposed as $x=(x_1,x_2)\in\R^{n_1}\times\R^{n_2}$
	 	and there exist set-valued mappings
	 	$\widetilde\Gamma_1\colon\R^{n_1}\tto\R^{\ell_1}$ and 
	 	$\widetilde\Gamma_2\colon\R^{n_2}\tto\R^{\ell_2}$
	 	as well as locally Lipschitz continuous functions
	 	$\gamma_1\colon\R^{n_2}\to\R^{\ell_1}$ and $\gamma_2\colon\R^{n_1}\to\R^{\ell_2}$
	 	such that
	 	\[
	 		\Gamma_1(x):=\widetilde\Gamma_1(x_1) +  \gamma_1(x_2)
	 		\qquad
	 		\Gamma_2(x):=\widetilde\Gamma_2(x_2)  + \gamma_2(x_1).
	 	\]
	 	The estimates from \cref{Thm:product_rule}
	 	can be specified via the decoupled sum rule,
	 	see \eqref{eq:Dec_Sum_SV}.
	\end{itemize}
\end{lemma}
\begin{proof}
    In case (i), \eqref{eq:CQ_product} follows from
    the Aubin property of $\Gamma_2$.
    Indeed, by the simple sum rule
    \cite[Corollary 3.10]{Mo18}, we obtain
    $D^*\Gamma_2(\xb,\zb_2)(0) \subset D^*\gamma(\xb)(0)$,
    where we exploited that $0 \in N_{\Omega}(\zb_2 - \gamma(\xb))$
    holds.
    Lipschitzness of $\gamma$ thus yields the claim.
 	For the proof of (ii), \eqref{eq:Dec_Sum_SV} yields
 	\begin{align*}
 		&D^*\Gamma_1((\bar x_1,\bar x_2),\bar z_1)(0)
 			\cap\bigl(-D^*\Gamma_2((\bar x_1,\bar x_2),\bar z_2)(0)\bigr)\\
 		&\qquad
 		\subset
 			\Bigl(D^*\widetilde\Gamma_1(\bar x_1,\bar z_1)(0)
 				\cap\bigl(-D^*\gamma_2(\bar x_1)(0)\bigr)\Bigr)
 			\times
 			\Bigl(D^*\gamma_1(\bar x_2)(0)
 				\cap\bigl(-D^*\widetilde{\Gamma}_2(\bar x_2,\bar z_2)(0)\bigr)\Bigr),
 	\end{align*}
	and due to $D^*\gamma_2(\bar x_1)(0)=\{0\}$ as well as $D^*\gamma_1(\bar x_2)(0)=\{0\}$,
 	which follows by local Lipschitz continuity of $\gamma_2$ and $\gamma_1$,
 	we obtain validity of the constraint qualification \eqref{eq:CQ_product}.
\end{proof}

In case where the single-valued functions appearing in the previous lemma
are continuously differentiable,
it is possible to obtain the product rule from 
\cref{lem:change_of_coordinates,lem:cartesian_products}
which yields slightly stronger results.
\begin{lemma}\label{lem:product_rule_via_change_of_coordinates}
	Fix $(\bar x,(\bar z_1,\bar z_2))\in\gph\Gamma$. Then the following assertions hold.
	\begin{enumerate}
		\item[(i)] In the setting of \cref{lem:CQ_product_naturally_given} (i),
		if $\gamma$ is continuously differentiable, then, exemplary,
			\[\begin{aligned}
				&D\Gamma(\bar x,(\bar z_1,\bar z_2))(u)&
				&\subset&
				&D\Gamma_1(\bar x,\bar z_1)(u)
					\times
					\bigl(\nabla\gamma(\bar x)u+T_\Omega(\bar z_2-\gamma(\bar x))\bigr),\\
				 &D^*\Gamma(\bar x,(\bar z_1,\bar z_2))(z_1^*,z_2^*)&
				&=&
				&\begin{cases}
					D^*\Gamma_1(\bar x,\bar z_1)(z_1^*)
					+\nabla\gamma(\bar x)^\top z_2^*
					&-z_2^*\in N_\Omega(\bar z_2-\gamma(\bar z))\\
					\varnothing	&\text{otherwise,}
				 \end{cases}
			\end{aligned}
			\]
			and the analogous estimate for the regular coderivative
			also holds with equality.
		\item[(ii)] In the setting of \cref{lem:CQ_product_naturally_given} (ii),
		if $\gamma_1$ and $\gamma_2$ are continuously differentiable, then we have equality in the formulas for the regular and limiting coderivative of $\Gamma$.
	\end{enumerate}
\end{lemma}
\begin{proof}
	We find a continuously differentiable function 
	$g\colon\R^n\times\R^{\ell_1}\times\R^{\ell_2}\to\R^{n+\ell_1+\ell_2}$ with
	invertible Jacobian and
	a closed set $C\subset\R^{n+\ell_1+\ell_2}$ such that $\gph\Gamma=g^{-1}(C)$
	holds in both situations.
	More precisely, in (i), we can choose
	\[
		g(x,z_1,z_2):=(x,z_1,z_2-\gamma(x))
		\qquad
		C:=\gph\Gamma_1\times\Omega,
	\]
	while in (ii), we make use of
	\[
		g((x_1,x_2),z_1,z_2):=(x_1,z_1-\gamma_1(x_2),x_2,z_2-\gamma_2(x_1))
		\qquad
		C:=\gph\Gamma_1\times\gph\Gamma_2.
	\]
	Thus, the results directly follow from \cref{lem:change_of_coordinates,lem:cartesian_products}.
\end{proof}

We want to finalize this section with a brief remark regarding the intersection rule from
generalized differentiation.
\begin{remark}\label{rem:intersection_rule}
	Observe that we have
	\[
		\Gamma^{-1}(z_1,z_2)
		=
		\{
			x\,|\,z_1\in\Gamma_1(x),\,z_2\in\Gamma_2(x)
		\}
		=
		\Gamma_1^{-1}(z_1)\cap\Gamma_2^{-1}(z_2)
	\]
	by definition of $\Gamma$.
	Thus, we can rewrite the estimates and qualification conditions for products from
	\cref{Thm:product_rule} as well as 
	\cref{lem:CQ_product_naturally_given,lem:product_rule_via_change_of_coordinates}
	in terms of intersections.
	Exemplary, let us mention that for some point $((\bar z_1,\bar z_2),\bar x)\in\gph\Gamma^{-1}$
	such that $M$ is calm at $((0,0,0,0),(\bar x,\bar z_1,\bar z_2,\bar x,\bar x))$, we have
	\begin{align*}
		D\Gamma^{-1}((\bar z_1,\bar z_2),\bar x)(v_1,v_2)
		&\subset
		D\Gamma_1^{-1}(\bar z_1,\bar x)(v_1)\cap D\Gamma_2^{-1}(\bar z_2,\bar x)(v_2)\\
		D^*\Gamma^{-1}((\bar z_1,\bar z_2),\bar x)(x^*)
		&\subset
		\bigcup\limits_{(x_1^*,x_2^*),\,x_1^*+x_2^*=x^*}
		D^*\Gamma_1^{-1}(\bar z_1,\bar x)(x_1^*)\times D^*\Gamma_2^{-1}(\bar z_2,\bar x)(x_2^*),
	\end{align*}
	and similar estimates can be obtained for the regular and the directional limiting 
	coderivative.
\end{remark}

\section{Applications}\label{sec:applications}

The importance of graphical derivative and limiting coderivative,
in particular in connection with isolated calmness and the Aubin property,
is well known, briefly mentioned in \cref{sec:introduction},
and also demonstrated in the previous two sections.
Well known is also the role of the limiting coderivative (limiting normal cone)
in the formulation of reasonable first-order necessary optimality conditions
for optimization problems, namely in form of so-called {\em M-stationarity} conditions.
Here, we present some applications
of the regular coderivative (regular subdifferential)
and the directional limiting coderivative.

\subsection{Optimality Conditions for Minimax Problems Based on Regular Subgradients}\label{sec:maxmin}

Given a continuously differentiable function
$\varphi\colon\R^n\times\R^m\to\R$ and a set-valued mapping
$G\colon\R^m \tto \R^n$ with closed graph, set
\[
	f(x,y):= \varphi(x,y) + \delta_{\gph G^{-1}}(x,y).
\]
Above, for some set $Q\subset\R^\ell$, $\delta_Q\colon \R^\ell\to\overline\R$ denotes the indicator function of $Q$ which vanishes on $Q$ and possesses value $\infty$ on $\R^\ell\setminus Q$.
Clearly, $f$ is a lower semicontinuous function due to
$\epi f = \epi \varphi \cap (\gph G^{-1} \times \R)$.

We consider the following optimization problem of minimax type
\begin{equation}\label{eq:max_min}\tag{MaxMin}
   \max\limits_{y\in\Omega}\min\limits_{x\in G(y)}\varphi(x,y)
\end{equation}
where $\Omega\subset\R^m$ is a closed set. 
Problems of this type arise frequently in game theory and can be interpreted as a
particular instance of bilevel optimization.
We refer the interested reader to \cite{Danskin1966,DemyanovMalozemov1974} for an introduction to minimax
programming and to the monographs \cite{Dempe2002,Mo18} for a detailed discussion of
bilevel optimization.

Using the function $f$ defined above and
the associated marginal function $\vartheta\colon\R^m\to\overline\R$ from 
\cref{sec:marginal_functions}, we easily see that
\begin{equation}\label{eq:surroagate_max_min}
	\min\limits_{y\in\Omega}(-\vartheta)(y)
\end{equation}
is a suitable surrogate problem of \eqref{eq:max_min}.
Subsequently, we exploit $S\colon\R^m\tto\R^n$ in order to denote the solution mapping
of the inner minimization problem $\min_x\{\varphi(x,y)\,|\,x\in G(y)\}$. 
Furthermore, we make use of the mapping $M\colon\R^m\times\R\tto\R^n$ given by
\[
	M(y,\alpha):=\{x\in G(y)\,|\,(x,y,\alpha)\in\epi\varphi\}.
\]
\begin{theorem}\label{thm:optimality_conditions_max_min}
	Let $(\xb,\yb)\in\R^n\times\R^m$ be a locally optimal solution of \eqref{eq:max_min}, i.e.,
	assume that there is a neighborhood $V$ of $\yb$ such that $\vartheta(\yb)\geq\vartheta(y)$
	holds for all $y\in\Omega\cap V$ while $\xb\in S(\yb)$ is valid. 
	Furthermore, let $M$ be inner calm* at $(\yb,\vartheta(\yb))$
	w.r.t.\ $\dom M$ in the fuzzy sense.
	Then the following condition holds:
	\[
		\Bigl[
			y^*\in \nabla_y\varphi(x,\yb)+\widehat D^*G(\yb,x)(\nabla_x\varphi(x,\yb))
			\ \forall x\in S(\yb)
		\Bigr]
		\ \Longrightarrow \
		y^*\in\widehat N_\Omega(\yb).
	\]
\end{theorem}
\begin{proof}
	The assumptions of the theorem guarantee that $\yb$ is a local minimizer of
	\eqref{eq:surroagate_max_min}. Thus, \cite[Theorem~4.3(i)]{MordukhovichNamYen2006}
	guarantees validity of 
	\begin{equation}\label{eq:regular_subdifferential_condition_max_min}
		\widehat\partial\vartheta(\yb)\subset\widehat{N}_\Omega(\yb).
	\end{equation}
	Due to the fuzzy inner calmness* of $M$, \cref{The:ValueFunction} guarantees
	\begin{align*}
		\widehat{\partial}\vartheta(\yb)
		&=
		\left\{
			y^*\,\middle|\,
			(0,y^*)\in\bigcap\nolimits_{x\in S(\yb)}\widehat{\partial}f(x,\yb)
		\right\}\\
		&=
		\left\{y^*\,\middle|\,
			(0,y^*)\in\bigcap\nolimits_{x\in S(\yb)}
				\left(\nabla\varphi(x,\yb)+\widehat{N}_{\gph G^{-1}}(x,\yb)\right)
		\right\}
	\end{align*}
	where we used the sum rule from \cite[Proposition~2.2]{MordukhovichNamYen2006} in
	order to decompose the regular subdifferential of $f$, and the equivalence of regular
	subgradients of the indicator function associated with a locally closed set as well as
	its regular normals, see \cite[Exercise~8.14]{RoWe98}. Elementary calculations now show 
	\[
		y^*\in\widehat{\partial}\vartheta(\yb)
		\quad\Longleftrightarrow\quad
		\bigl(y^*-\nabla_y\varphi(x,\yb),-\nabla_x\varphi(x,\yb)\bigr)
		\in\widehat{N}_{\gph G}(\yb,x)
		\quad
		\forall x\in S(\yb).
	\]
	Exploiting the definition of the regular coderivative as well as
	\eqref{eq:regular_subdifferential_condition_max_min}, the result follows. 
\end{proof}

Invoking \cref{cor:IC*_via_FOSCclm} and \eqref{eq:simpler_CQ_marginal_functions}
as well as the calculus rules for the directional limiting subdifferential from
\cite{BeGfrOut19}, 
we observe that $M$ is inner calm* in the fuzzy sense
w.r.t.\ $\dom M$ at $(\yb,\vartheta(\yb))$ whenever it is inner semicompact there while
the condition
\begin{align*}
	&\forall u\in DG(\yb,x)(0)\setminus\{0\}\colon\\
	&\qquad \nabla_x\varphi(x,\yb)u\leq 0\
	\ \Longrightarrow \
		\nabla_y\varphi(x,\yb)+D^*G((\yb,x);(0,u))(\nabla_x\varphi(x,\yb))
		\subset\{0\}
\end{align*}
holds for each $x\in S(\yb)$.

In order to infer a necessary optimality condition in terms of initial
problem data from \eqref{eq:regular_subdifferential_condition_max_min}, it is indeed
necessary to find a reasonable \emph{lower} estimate for $\widehat{\partial}\vartheta(\yb)$.
Here, our novel results from \cref{The:ValueFunction}, which yield an equivalent representation
of this subdifferential and, thus, an optimality condition of reasonable strength, work
quite nicely. In practice, it may happen that $\widehat{\partial}\vartheta(\yb)$ is empty
for points $(\xb,\yb)$ which are locally optimal to \eqref{eq:max_min}, and in this case,
\eqref{eq:regular_subdifferential_condition_max_min} holds trivially. However, one could still
use \cref{thm:optimality_conditions_max_min} in order to exclude points which do not provide
locally optimal solutions of \eqref{eq:max_min} in this case. 

In order to infer necessary optimality conditions in terms of limiting normals
for \eqref{eq:max_min}, one could presume local Lipschitzness of $\vartheta$ around the 
point of interest $\yb$ so that \cite[Theorem~8.15]{RoWe98} yields 
$0\in\partial(-\vartheta)(\yb)+N_\Omega(\yb)$.
Now, to apply our results from \cref{The:ValueFunction}, we need to convexify the limiting
subdifferential in order to get rid of the negative sign, i.e., we have to evaluate
$0\in-\conv\partial\vartheta(\yb)+N_\Omega(\yb)$ which is possible now with the aid of
Carath\'{e}odory's theorem whenever $M$ from above is inner semicompact at $(\yb,\vartheta(\yb))$.

\subsection{Semismoothness* Calculus}\label{sec:semismoothness}

In the recent paper \cite{GfrOut2019}, a new notion of {\em semismoothness*} for sets and set-valued mappings has been introduced
and used to propose a semismooth* Newton method for the solution of
generalized equations.
This notion is based on the directional limiting constructions
and so our results yield calculus rules for semismoothness*,
i.e., the rules that describe how this property can be transferred
from one or more objects to another by the transformations discussed
in \Cref{sec:recovering_calculus}.

Following \cite[Definition 3.1]{GfrOut2019},
we call a set $\Omega \subset \R^n$ semismooth* at $\zb \in \Omega$
if for all $w \in \R^n$, it holds
\[	
	\skalp{z^*,w}=0 \quad \forall z^* \in N_{\Omega}(\zb;w).
\]
A set-valued mapping $M\colon\R^m \tto \R^n$ is called semismooth*
at $(\yb,\xb) \in \gph M$ if $\gph M$ is semismooth* at $(\yb,\xb)$,
i.e., for all $(v, u) \in \R^m \times \R^n$, we have
\[
	\skalp{y^*,v}=\skalp{x^*,u} 
	\quad 
	\forall (x^*,y^*) \in \gph D^*M((\yb,\xb);(v,u)).
\]

\begin{theorem}
    Let $M\colon\R^m \tto \R^n$ be a set-valued mapping.
    \begin{itemize}
     \item[(i)] Suppose that $\yb\in\dom M$ is chosen such that $\gph M$ is locally closed around 
     $\{\yb\}\times\R^n$ and $\dom M$ is locally closed around $\yb$.
     If $M$ is semismooth* at $(\yb,\xb)$ for each $\xb \in M(\yb)$
     as well as inner calm* at $\yb$ w.r.t.\ $\dom M$,
     then $\dom M$ is semismooth* at $\yb$.
     \item[(ii)] Suppose that $M$ is locally closed around
     $(\yb,\xb) \in \gph M$.
     If $M$ is semismooth* as well as calm at $(\yb,\xb)$,
     then $M(\yb)$ is semismooth* at $\xb$.
    \end{itemize}
\end{theorem}
\begin{proof}
 To show (i), pick $v\in\R^m$ and $y^* \in N_{\dom M}(\yb;v)$.
 By \cref{The : IC*calc}, we infer that there exist $\xb \in M(\yb)$
 and $u \in DM(\yb,\xb)(v)$ such that
 $y^* \in D^*M((\yb,\xb);(v,u))(0)$ holds.
 Semismoothness* of $M$ at $(\yb,\xb)$ 
 now readily yields $\skalp{y^*,v}=\skalp{0,u} = 0$.
 
 To show (ii), let $u\in\R^n$ and $x^* \in N_{M(\yb)}(\xb;u)$ be given.
 \Cref{The : Ccalc} implies the existence of
 $y^* \in D^*M((\bar y,\bar x);(0,u))(-x^*)$
 and, hence, from semismoothness* of $M$ at $(\yb,\xb)$, we get 
 $\skalp{x^*,u}=-\skalp{y^*,0} = 0$.
\end{proof}

Naturally, one can derive similar results in more
useful situations described in \cref{sec:recovering_calculus}.
That is to say, in the settings of standard calculus rules,
semismoothness* can be transferred to the desired object
not only from $M$, but also from (semismoothness* of) the input data.
Exemplary, semismoothness* of the image representation mapping
$M_3(y):=g^{-1}(y) \cap C$, see \cref{sec:elementary_calculus},
is implied by semismoothness* of $g\colon\R^n\to\R^n$ and $C\subset\R^n$, provided $g$ is
Lipschitz continuous. Hence, semismoothness* is transferred
from $g$ and $C$, via $M_3$, to $g(C)=\dom M_3$.

\section{Concluding Remarks}\label{sec:conclusions}

In this paper, we have seen that the two enhanced estimates stated in
\Cref{The : IC*calc,The : Ccalc} are enough to recover the major calculus
rules of variational analysis regarding tangents as well as regular, limiting,
and directional limiting normals under mild calmness-type conditions. As a
by-product, we obtained calculus rules for the computation of generalized
derivatives associated with Cartesian products of set-valued mappings in
\cref{sec:product_rule}.
In \cref{Sec:SC}, we interrelated the different (inner) calmness-type conditions
which were used to establish the two patterns from \Cref{The : IC*calc,The : Ccalc}.
Particularly, we have shown that Gfrerer's \emph{first-order sufficient condition for
calmness} can be used to guarantee the validity of inner calmness* in the fuzzy
sense. It remains an open question whether the former condition already yields
inner calmness*.
Furthermore, a precise study on the different moduli of the calmness-type
properties seems to be 
an interesting subject of future research.

\section*{Acknowledgments}
The research of the first author was supported by the Austrian Science Fund
(FWF) under grants P29190-N32 and P32832-N.
The authors are grateful to Ji\v{r}\'{i} V. Outrata, who suggested
the application to semismoothness* from \cref{sec:semismoothness}.

%\bibliographystyle{abbrv}
%\bibliography{references}

\begin{thebibliography}{10}

\bibitem{AdamHenrionOutrata2018}
L.~Adam, R.~Henrion, and J.~Outrata.
\newblock On {M}-stationarity conditions in {M}{P}{E}{C}s and the associated
  qualification conditions.
\newblock {\em Mathematical Programming}, 168(1):229--259, 2018.

\bibitem{AubinEkeland1984}
J.-P. Aubin and I.~Ekeland.
\newblock {\em Applied Nonlinear Analysis}.
\newblock Wiley-Interscience, New York, 1984.

\bibitem{AubinFrankowska2009}
J.-P. Aubin and H.~Frankowska.
\newblock {\em Set-Valued Analysis}.
\newblock Birkh{\"a}user, Boston, 2009.

\bibitem{Be19}
M.~Benko.
\newblock On inner calmness{$^*$}, generalized calculus, and derivatives of the
  normal cone mapping.
\newblock {\em arXiv}, pages 1--27, 2019.

\bibitem{BenkoCervinkaHoheisel2019}
M.~Benko, M.~{\v{C}}ervinka, and T.~Hoheisel.
\newblock Sufficient conditions for metric subregularity of constraint systems
  with applications to disjunctive and ortho-disjunctive programs.
\newblock {\em Set-Valued and Variational Analysis}, pages 1--35, 2021.

\bibitem{BeGfrOut19}
M.~Benko, H.~Gfrerer, and J.~V. Outrata.
\newblock Calculus for directional limiting normal cones and subdifferentials.
\newblock {\em Set-Valued and Variational Analysis}, 27(3):713--745, 2019.

\bibitem{BeGfrOut18}
M.~Benko, H.~Gfrerer, and J.~V. Outrata.
\newblock Stability analysis for parameterized variational systems with
  implicit constraints.
\newblock {\em Set-Valued and Variational Analysis}, 28(1):167--193, 2020.

\bibitem{CiFaKru18}
R.~Cibulka, M.~Fabian, and A.~Y. Kruger.
\newblock On semiregularity of mappings.
\newblock {\em Journal of Mathematical Analysis and Applications},
  473(2):811--836, 2019.

\bibitem{CaGiHePa19}
M.~Cánovas, M.~Gisbert, R.~Henrion, and J.~Parra.
\newblock Lipschitz lower semicontinuity moduli for linear inequality systems.
\newblock {\em Journal of Mathematical Analysis and Applications},
  490(2):1--21, 2020.

\bibitem{Danskin1966}
J.~M. Danskin.
\newblock The theory of max-min, with applications.
\newblock {\em SIAM Journal on Applied Mathematics}, 14(4):641--664, 1966.

\bibitem{Dempe2002}
S.~Dempe.
\newblock {\em Foundations of Bilevel Programming}.
\newblock Kluwer, Dordrecht, 2002.

\bibitem{DemyanovMalozemov1974}
V.~F. Demyanov and V.~V. Malozemov.
\newblock {\em Introduction to Minimax}.
\newblock Wiley, New York, 1974.

\bibitem{DoRo04}
A.~L. Dontchev and R.~T. Rockafellar.
\newblock Regularity and conditioning of solution mappings in variational
  analysis.
\newblock {\em Set-Valued Analysis}, 12:79--109, 2004.

\bibitem{DoRo14}
A.~L. Dontchev and R.~T. Rockafellar.
\newblock {\em Implicit Functions and Solution Mappings}.
\newblock Springer, Heidelberg, 2014.

\bibitem{FabHenKruOut10}
M.~J. Fabian, R.~Henrion, A.~Y. Kruger, and J.~V. Outrata.
\newblock Error bounds: necessary and sufficient conditions.
\newblock {\em Set-Valued and Variational Analysis}, 18(2):121--149, 2010.

\bibitem{Gfr13a}
H.~Gfrerer.
\newblock On directional metric regularity, subregularity and optimality
  conditions for nonsmooth mathematical programs.
\newblock {\em Set-Valued and Variational Analysis}, 21(2):151--176, 2013.

\bibitem{Gfrerer2014}
H.~Gfrerer.
\newblock Optimality conditions for disjunctive programs based on generalized
  differentiation with application to mathematical programs with equilibrium
  constraints.
\newblock {\em SIAM Journal on Optimization}, 24(2):898--931, 2014.

\bibitem{GfrKl16}
H.~Gfrerer and D.~Klatte.
\newblock Lipschitz and {H}\"{o}lder stability of optimization problems and
  generalized equations.
\newblock {\em Mathematical Programming}, 158:35--75, 2016.

\bibitem{GfrOut16}
H.~Gfrerer and J.~V. Outrata.
\newblock On {L}ipschitzian properties of implicit multifunctions.
\newblock {\em SIAM Journal on Optimization}, 26(4):2160--2189, 2016.

\bibitem{GfrOut2019}
H.~Gfrerer and J.~V. Outrata.
\newblock On a semismooth$^*$ {N}ewton method for solving generalized
  equations.
\newblock {\em SIAM Journal on Optimization}, 31(1):489--517, 2021.

\bibitem{GfrererYe2017}
H.~Gfrerer and J.~J. Ye.
\newblock New constraint qualifications for mathematical programs with
  equilibrium constraints via variational analysis.
\newblock {\em SIAM Journal on Optimization}, 27(2):842--865, 2017.

\bibitem{GfrererYe2020}
H.~Gfrerer and J.~J. Ye.
\newblock New sharp necessary optimality conditions for mathematical programs
  with equilibrium constraints.
\newblock {\em Set-Valued and Variational Analysis}, 28(2):395--426, 2020.

\bibitem{GinchevMordukhovich2011}
I.~Ginchev and B.~S. Mordukhovich.
\newblock On directionally dependent subdifferentials.
\newblock {\em Proceedings of the Bulgarian Academy of Sciences},
  64(4):497--508, 2011.

\bibitem{HarderWachsmuth2018}
F.~Harder and G.~Wachsmuth.
\newblock The limiting normal cone of a complementarity set in {S}obolev
  spaces.
\newblock {\em Optimization}, 67(10):1579--1603, 2018.

\bibitem{HenJouOut02}
R.~Henrion, A.~Jourani, and J.~V. Outrata.
\newblock On the calmness of a class of multifunctions.
\newblock {\em SIAM Journal on Optimization}, 13(2):603--618, 2002.

\bibitem{HenOut05}
R.~Henrion and J.~V. Outrata.
\newblock Calmness of constraint systems with applications.
\newblock {\em Mathematical Programming}, 104(1):437--464, 2005.

\bibitem{Iof79a}
A.~D. Ioffe.
\newblock Necessary and sufficient conditions for a local minimum. 1: {A}
  reduction theorem and first order conditions.
\newblock {\em SIAM Journal on Control and Optimization}, 17(2):245--250, 1979.

\bibitem{Io17}
A.~D. Ioffe.
\newblock {\em Variational Analysis of Regular Mappings}.
\newblock Springer, Cham, 2017.

\bibitem{IofOut08}
A.~D. Ioffe and J.~V. Outrata.
\newblock On metric and calmness qualification conditions in subdifferential
  calculus.
\newblock {\em Set-Valued Analysis}, 16(2):199--227, 2008.

\bibitem{IofPen96}
A.~D. Ioffe and J.-P. Penot.
\newblock Subdifferentials of performance functions and calculus of
  coderivatives of set-valued mappings.
\newblock {\em Serdica Mathematical Journal}, 22(3):359--384, 1996.

\bibitem{Kl87}
D.~Klatte.
\newblock A note on quantitative stability results in nonlinear programming.
\newblock {\em Seminarbericht, Sektion Mathematik, Humboldt-Universit\"{a}t zu
  Berlin}, 90:77--86, 1987.

\bibitem{Kl94}
D.~Klatte.
\newblock On quantitative stability for non-isolated minima.
\newblock {\em Control and Cybernetics}, 23:183--200, 1994.

\bibitem{KlKum02}
D.~Klatte and B.~Kummer.
\newblock {\em Nonsmooth Equations in Optimization}.
\newblock Kluwer Academic, Dordrecht, 2002.

\bibitem{KlKu15}
D.~Klatte and B.~Kummer.
\newblock On calmness of the argmin mapping in parametric optimization
  problems.
\newblock {\em Journal of Optimization Theory and Applications},
  165(3):708--719, 2015.

\bibitem{Levy96}
A.~B. Levy.
\newblock Implicit multifunction theorems for the sensitivity analysis of
  variational conditons.
\newblock {\em Mathematical Programming}, 74:333--350, 1996.

\bibitem{LongWangYang2017}
P.~Long, B.~Wang, and X.~Yang.
\newblock Calculus for directional coderivatives and normal cones in {A}splund
  spaces.
\newblock {\em Positivity}, 21:1115--1142, 2017.

\bibitem{Mehlitz2019a}
P.~Mehlitz.
\newblock On the sequential normal compactness condition and its
  restrictiveness in selected function spaces.
\newblock {\em Set-Valued and Variational Analysis}, 27(3):763--782, 2019.

\bibitem{MehlitzWachsmuth2018}
P.~Mehlitz and G.~Wachsmuth.
\newblock The limiting normal cone to pointwise defined sets in {L}ebesgue
  spaces.
\newblock {\em Set-Valued and Variational Analysis}, 26(3):449--467, 2018.

\bibitem{Mordukhovich1976}
B.~S. Mordukhovich.
\newblock Maximum principle in the problem of time optimal response with
  nonsmooth constraints.
\newblock {\em Journal of Applied Mathematics and Mechanics}, 40(6):960--969,
  1976.

\bibitem{Mo06}
B.~S. Mordukhovich.
\newblock {\em Variational Analysis and Generalized Differentiation, Vol.\ I:
  Basic Theory, Vol.\ II: Applications}.
\newblock Springer, Berlin, 2006.

\bibitem{Mo18}
B.~S. Mordukhovich.
\newblock {\em Variational Analysis and Applications}.
\newblock Springer, Cham, 2018.

\bibitem{MordukhovichNamYen2006}
B.~S. Mordukhovich, N.~M. Nam, and N.~D. Yen.
\newblock Fr{\'{e}}chet subdifferential calculus and optimality conditions in
  nondifferentiable programming.
\newblock {\em Optimization}, 55(5--6):685--708, 2006.

\bibitem{Nechita2012}
D.-M. Nechita.
\newblock About some links between the {D}ini--{H}adamard-like normal cone and
  the contingent cone.
\newblock {\em Studia Universitatis Babe{\c{s}}--Bolyai Mathematica},
  57(4):541--549, 2012.

\bibitem{Rob81}
S.~M. Robinson.
\newblock Some continuity properties of polyhedral multifunctions.
\newblock In H.~K{\"o}nig, B.~Korte, and K.~Ritter, editors, {\em Mathematical
  Programming at {O}berwolfach}, pages 206--214. Springer, Berlin, 1981.

\bibitem{Ro70}
R.~T. Rockafellar.
\newblock {\em Convex Analysis}.
\newblock Princeton University Press, Princeton, 1970.

\bibitem{RoWe98}
R.~T. Rockafellar and R.~J.-B. Wets.
\newblock {\em Variational Analysis}.
\newblock Springer, Berlin, 1998.

\bibitem{YeZhou2018}
J.~J. Ye and J.~Zhou.
\newblock Verifiable sufficient conditions for the error bound property of
  second-order cone complementarity problems.
\newblock {\em Mathematical Programming}, 171(1):361--395, 2018.

\end{thebibliography}

\end{document}